\documentclass[letterpaper, 11pt]{amsart}
    \usepackage{amsmath, amssymb, amsfonts}
    \usepackage{mathtools}

    \usepackage{dsfont}
    \usepackage{mathrsfs}
    \usepackage{cancel}
    
    \usepackage[top=1in,left=1in,right=1in,bottom=1in]{geometry}
    \usepackage{setspace}
    \usepackage[utf8x]{inputenc}
    \usepackage[T1]{fontenc}
    \usepackage{lmodern}
    \usepackage{graphicx}
    \usepackage{enumerate}
    \usepackage[dvipsnames]{xcolor}
    \usepackage{ifthen} 
    
    \usepackage{tikz} 
    \usepackage{tikz-cd}
    \usepackage{verbatim} 
    \usepackage[pdftex,  colorlinks=true]{hyperref}
    \hypersetup{urlcolor=RoyalBlue,  linkcolor=RedOrange,  citecolor=black}
   
    \usepackage[normalem]{ulem} 

    \usepackage{marginnote}		

    \usepackage{amsthm}

    \usepackage{thmtools}
    \usepackage{thm-restate}    
     \usepackage[nameinlink]{cleveref}   
     \crefname{axiom}{axiom}{axioms}			
     \crefname{claim}{claim}{claims}			
     
    \newtheorem{theorem}{Theorem}[section]
    \newtheorem{lemma}[theorem]{Lemma}
    
    \newtheorem{coro}[theorem]{Corollary}
    \newtheorem{prop}[theorem]{Proposition}

    \newtheorem{defn}[theorem]{Definition}

    \newtheorem{obs}[theorem]{Observation}

    \newtheorem{question}[theorem]{Question}

   \theoremstyle{remark}	
    \newtheorem{remark}[theorem]{Remark} 

    \newtheorem{ex}[theorem]{Example}
    
    \newcounter{thmcount}
\newcommand*{\numberedtheorem}[3]{\theoremstyle{plain}\newtheorem*{makethm\thethmcount}{#1}
    \ifthenelse{\equal{#2}{}}{\begin{makethm\thethmcount}#3\end{makethm\thethmcount}\stepcounter{thmcount}}
    {\begin{makethm\thethmcount}[#2]#3\end{makethm\thethmcount}\stepcounter{thmcount}}}

    \newcommand{\mf}[1]{\mathfrak{#1}}

    \newcommand{\F}{\mathbb{F}}

    \newcommand{\R}{\mathbb{R}}

    \newcommand{\Z}{\mathbb{Z}}

    \newcommand{\om}{\omega}

    \newcommand{\ra}{\rightarrow}
    
    \renewcommand{\iff}{\Leftrightarrow}

    \newcommand{\bigset}[1]{\mathrm{Big}\pns{#1}}
    \newcommand{\diam}{\mathrm{diam}}
    \newcommand{\dist}{d}

    \newcommand{\Isom}{\text{Isom}}

    \newcommand{\tand}{\text{ and }}

    \newcommand{\stab}[1]{\text{Stab}\pns{#1}}
    \newcommand{\Sym}{\operatorname{Sym}}
    
    \newcommand{\ceil}[1]{\left\lceil #1 \right\rceil}
    
    \newcommand{\pns}[1]{\left( #1 \right)}

    \newcommand{\braces}[1]{\left\{ #1 \right\}}
    \newcommand{\abs}[1]{\left| #1 \right|}
    \newcommand{\abrackets}[1]{\left\langle #1 \right\rangle}

    \renewcommand{\bar}{\overline}
    \newcommand{\defeq}{\vcentcolon=}
    
     \newcommand*{\mc}[1]{\mathcal{#1}}
     
     \def\gener{X}

     \def\nest{\sqsubseteq}
     \def\trans{\pitchfork}
     
     \def\isom{\cong}
     \makeatletter
	 \def\squiggly{\bgroup \markoverwith{\textcolor{red}{\lower3.5\p@\hbox{\sixly \char58}}}\ULon}
	 \makeatother
	 \def\PiX{\prod_{W \in \mf{S}}2^{\fontact W}}

    \newcommand{\B}{\operatorname{Big}}
    \newcommand{\X}{\mathcal X}
    \newcommand{\s}{\mathfrak S}
    
    \newcommand{\fontact}{{\mathcal C}}
    \newcommand{\Aut}{\operatorname{Aut}}
    \newcommand{\cuco}[1]{{\mathcal #1}}

    \newcommand{\orth}{\bot}
    \usepackage{mathabx}
\newcommand{\propnest}{\sqsubsetneq}

\newcommand{\tsh}[1]{\left\{\kern-.7ex\left\{#1\right\}\kern-.7ex\right\}}
\newcommand{\Tsh}[2]{\tsh{#2}_{#1}}
\newcommand{\ignore}[2]{\Tsh{#2}{#1}}

\newcommand{\cal}{\mathcal}
    
    \newcommand{\domains}{\mathfrak{S}}
    \newcommand{\genset}{\gener}
    \newcommand{\displacement}[2]{\tau_{#1} \pns{ #2 }}
    \newcommand{\displacementBound}{\tau_0}
    \newcommand{\Ball}[1]{X^{#1}}	
    \renewcommand{\hat}{\widehat}

\theoremstyle{definition}
\newtheorem{definition}[theorem]{Definition}

\begin{document}

   \title{Hierarchically hyperbolic groups and uniform exponential growth}
   
   \author{Carolyn R. Abbott}
	\address{Department of Mathematics, Brandeis University, Waltham, MA 02453}
	\email{carolynabbott@brandeis.edu}
	\author{Thomas A. Ng}
	\address{Mathematics Department, Technion -- Israel Institute of Technology, Haifa, 32000 Israel}
	\email{thomas.ng@campus.technion.ac.il}
	\author{Davide Spriano}
	\address{Mathematical Institute, University of Oxford, Oxford OX2 6GG}
	\email{spriano@maths.ox.ac.uk}

	\date{}
	
	        \maketitle
	        
        \centerline{
\textit{
With an appendix by Radhika Gupta and Harry Petyt}
}

\begin{abstract}
	We give several sufficient conditions for uniform exponential growth in the setting of virtually torsion-free hierarchically hyperbolic groups.  For example, any hierarchically hyperbolic group that is also acylindrically hyperbolic has uniform exponential growth.
	In addition, we provide a quasi-isometric characterization of hierarchically hyperbolic groups without uniform exponential growth. 
	To achieve this, we gain new insights on the structure of certain classes of hierarchically hyperbolic groups. 
	Our methods give a new unified proof of uniform exponential growth for several examples of groups with notions of non-positive curvature. 
	In particular, we obtain the first proof of uniform exponential growth for certain groups that act geometrically on CAT(0) cubical spaces of dimension 3 or more.
	Under additional hypotheses, we show that a quantitative Tits alternative holds for hierarchically hyperbolic groups.  
\end{abstract}

\section{Introduction}

A finitely generated group has \emph{(uniform) exponential growth} if the number of elements that can be spelled with words of bounded length grows (uniformly) exponentially fast with respect to \emph{any} finite generating set.  Exponential growth rates and uniform exponential growth rates are of interest in a broad range of areas, including differential geometry, dynamical system theory, and the theory of unitary representations (see \cite{Harpe} and citations therein). 

Gromov asked if every finitely generated group with exponential growth has uniform exponential growth.  However, this is not the case: the first example of a group with exponential growth but not uniform exponential growth was constructed by Wilson \cite{Wilson}, and additional counterexamples have since been constructed \cite{Wilson2,Bartholdi,Nek}.  However, Gromov's question is still open for finitely presented groups.

Many classes of groups are known to either be virtually nilpotent or have uniform exponential growth.  
This form of growth gap was shown for linear groups by Eskin, Mozes, and Oh \cite{EMO}; for hyperbolic groups by Koubi \cite{Koubi}; for fundamental groups of manifolds with pinched negative curvature by Besson, Coutois, and Gallot \cite{BCG}; for finitely generated subgroups of the mapping class group by Mangahas \cite{Mangahas}
and, more generally, automorphism groups of one-ended hyperbolic groups by Kropholler, Lyman, and Ng \cite{KrophollerLymanNg}; for linearly growing subgroups of $\operatorname{Out}(F_n)$ by Bering \cite{Bering}; and for groups acting without global fixed points on 2--dimensional CAT(0) cube complexes, with some generalizations to higher dimensions, by work of Gupta, Jankiewicz, and Ng \cite{GJN:CubesUEG}.  We note that the full mapping class group was known to have uniform exponential growth (via its action on homology) by Anderson, Aramayona, and Shackleton \cite{AAS}, and the torsion-free case of 2-dimensional cubical groups was shown by Kar and Sageev \cite{KarSageev}.

In the case of torsion-free cubical groups, it is remarkable that all known proofs rely heavily on the assumption that the cube complex has low dimension or that it has isolated flats, a strong form of relative hyperbolicity.  
Indeed, the authors are not aware of any general proof that works in dimensions higher than 2.  
Moreover, some of the more curious cubical groups do not act geometrically on CAT(0) cube complexes with isolated flats.  For example, the genus 2 handlebody group acts geometrically on a CAT(0) cube complex \cite{HamenstadtHensel:Handlebody2},  but Dehn twist flats can intersect along infinite subgroups, so this cube complex cannot have isolated flats.  

Our main contribution, a combination of the main body of the paper and the appendix, is the following.

\begin{theorem}\label{thm_main_and_appendix}
Let $G$ be a group virtually acting freely and cocompactly on a locally finite, finite-dimensional CAT(0) cube complex $X$, and assume that $X$ has a factor system. Then either $G$ has uniform exponential growth or $G$ is virtually abelian.
\end{theorem}

Requiring that $X$ has a factor system should not be thought of as a restrictive hypothesis, as there is no known example of a CAT(0) cube complex with a cocompact group action that does not admit a factor system.  The existence of a factor system is a natural condition with several consequences. 
For example, Durham, Hagen, and Sisto give an alternative proof of rank-rigidity for CAT(0) cube complexes with a factor system \cite[Corollary~9.24]{DurhamHagenSisto:Boundaries}, a result originally proved by Caprace and Sageev \cite{capracesageev:rank}.
Any group acting geometrically on a CAT(0) cube complex with a factor systems is a \emph{hierarchically hyperbolic group} \cite[Theorem~A]{HagenSusse}. This is a starting point for results such as \Cref{thm_main_and_appendix} and \cite[Corollary~9.24]{DurhamHagenSisto:Boundaries}.  Beyond hierarchical hyperbolicity, factor systems also appear in work of of Incerti-Medici and Zalloum \cite{IncertiZalloum} and implicitly in work of Genevois \cite{genevois:hyperbolicities}.

The main technical result of this paper is an analogous version of \Cref{thm_main_and_appendix} for the larger class of hierarchically hyperbolic groups. 
This is a large class of groups introduced by Behrstock, Hagen, and Sisto \cite{BehrstockHagenSisto:HHS2} whose  structure is similar to that of mapping class groups and CAT(0) cubical groups. This class of groups includes hyperbolic groups, mapping class groups, many (conjecturally all) CAT(0) cubical groups, fundamental groups of most 3--manifolds, and various combinations of the above groups, including direct products, certain quotients, and graph products \cite{BehrstockHagenSisto:HHS2,BerlaiRobbio}.

Hierarchically hyperbolic groups and, more generally, hierarchically hyperbolic spaces are defined axiomatically, generalizing the Masur--Minsky machinery for mapping class groups \cite{MM}.  Roughly speaking, a \emph{hierarchically hyperbolic space} (HHS) consists of a metric space $\mathcal X$ along with the following data: an index set $\s$ of \emph{domains} with three relations (nesting, transversality, and orthogonality), $\delta$--hyperbolic spaces $\fontact U$ associated to each domain $U\in \s$, and projection maps $\X\to\fontact U$ and $\fontact U\to\fontact V$ (defined for certain $U,V\in \s$) satisfying certain conditions.  We denote this entire package of information by $(\mathcal X,\s)$.    
In some sense, this set of hyperbolic spaces can be thought of as a set of coordinate spaces. We are used to understanding the space \(\mathbb{R}^n\) by associating to a point a \(n\)--tuple of elements of \(\mathbb{R}\), which is a hyperbolic space. A simplistic but useful viewpoint on hierarchically hyperbolic space is to think of the hierarchical structure as nothing but a more complicated coordinate system on the metric space $\X$.   We discuss this in more detail in \Cref{sec:coordsys}.
A \emph{hierarchically hyperbolic group} (HHG) is essentially a group whose Cayley graph is an HHS such that the action of the group on the Cayley graph is compatible with the HHS structure; we use $(G,\s)$ to denote a HHG.  

The following is a structure theorem for virtually torsion-free hierarchically hyperbolic groups providing a sufficient condition for uniform exponential growth.

\begin{theorem} \label{thm:mainthm} Let $(G,\s)$ be a virtually torsion-free hierarchically hyperbolic group.  Then either $G$ has uniform exponential growth, or there is a space \(E\) such that the Cayley graph of $G$ is quasi-isometric to \(\mathbb{Z} \times E\). 
\end{theorem}

We note that the two possible outcomes in the theorem are not mutually exclusive: a simple example is given by the group \(\mathbb{Z}\times F_2\), where $F_2$ is a free group of rank two. Such a group is clearly a product of the form $\Z\times E$, but it has uniform exponential growth because it surjects onto $ F_2$.

When $G$ is a group virtually acting freely and cocompactly on a locally finite, finite-dimensional CAT$(0)$ cube complex with a factor system and $G$ is not directly decomposable, then \Cref{thm_main_and_appendix} follows immediately from \Cref{thm:mainthm}.  In the case that $G$ is directly decomposable, \Cref{thm_main_and_appendix} follows from applying \Cref{thm:mainthm} and \Cref{thm:main} from the appendix. 

\subsection{HHG with uniform exponential growth}
The first consequence of \Cref{thm:mainthm} is that if the Cayley graph of a hierarchically hyperbolic group \(G\) is not quasi-isometric to a (nontrivial) product, then \(G\) has uniform exponential growth.  We  state several corollaries giving conditions under which this is the case.   There is significant overlap in the situations covered by these corollaries: our goal is simply to highlight a wide variety of conditions that imply uniform exponential growth.

A subset $Y$ of a metric space $X$ is \emph{quasi-convex} if every $(\lambda,\varepsilon)$--quasi-geodesic in $X$ with endpoints on $Y$ is contained in a uniform neighborhood (depending on $\lambda,\varepsilon$) of the subset $Y$.  Such a subspace is sometimes referred to as Morse \cite{Cordes, CharneySultan}, strongly quasi-convex \cite{Tran}, or quasi-geodesically quasi-convex \cite{RussellSprianoTran}. In particular, if $Y$ is a quasi-geodesic satisfying this property, it is typically called Morse.

\begin{restatable*}{coro}{coroAsyCone} 
\label{cor:ac}
Every non-virtually cyclic virtually torsion-free hierarchically hyperbolic group which has an asymptotic cone containing a cut-point has uniform exponential growth.  In particular, if the Cayley graph of a virtually torsion-free hierarchically hyperbolic group $G$ contains an unbounded Morse quasi-geodesic, then $G$ has uniform exponential growth.
\end{restatable*}

One particularly nice class of hierarchically hyperbolic groups to which \Cref{cor:ac} can be applied is those which are \emph{acylindrically hyperbolic}.  The action of a group $G$ on a metric space $X$ is \emph{acylindrical} if for all $\varepsilon > 0$ there exist constants $R,N \geq 0$ such that for all $x,y\in X$ with $d(x,y)\geq R$, 
\[
	\#\{g\in G\mid d(x,gx)\leq \varepsilon \text{ and } d(y,gy)\leq \varepsilon \}\leq N.
\]  A group is \emph{acylindrically hyperbolic} if it admits a non-elementary acylindrical action on a hyperbolic space, that is, such that the limit set of the action contains at least three points.\footnote{Equivalently, a group is acylindrically hyperbolic if it is not virtually cyclic and admits an acylindrical action on a hyperbolic space with unbounded orbits.}  It is unknown if \emph{all} acylindrically hyperbolic groups have uniform exponential growth.  However, it follows from Sisto \cite{Sisto} that every acylindrically hyperbolic group contains an infinite order Morse element, that is, an infinite order element $g$ such that the quasi-geodesic $\langle g\rangle$ in the Cayley graph of $G$ is Morse. Thus we immediately obtain the following result.  

\begin{coro}\label{cor:ah}
Virtually torsion-free hierarchically hyperbolic groups which are acylindrically hyperbolic have uniform exponential growth.
\end{coro}

The following gives another way of using quasi-convex subspaces to determine that $G$ is not quasi-isometric to a product with unbounded factors.  
\begin{restatable*}{coro}{coroQuasiConvex}
\label{cor:qcx}
Every virtually torsion-free hierarchically hyperbolic group which is not virtually cyclic and contains an infinite quasi-convex subgroup of infinite index has uniform exponential growth. 
\end{restatable*}

For any hierarchically hyperbolic space $(\X,\s)$, the index set $\s$ contains a domain which is largest under the nesting relation; we will always denote this domain $S$ and its associated hyperbolic space $\fontact S$.  Given a hierarchically hyperbolic group, we can use the geometry of the hyperbolic space $\fontact S$ to determine that $G$ is not quasi-isometric to a product with unbounded factors. 

\begin{restatable*}{coro}{coroNonelem}
\label{cor:CSnonelem}	
Let $(G,\s)$ be a virtually torsion-free hierarchically hyperbolic group such that $\fontact S$ is a non-elementary hyperbolic space.  Then $G$ has uniform exponential growth.
\end{restatable*}

Under the assumptions of \Cref{cor:CSnonelem}, we actually obtain more information than what is stated in  \Cref{thm:mainthm}.  We can additionally  show that $G$ satisfies a \emph{quantitative Tits alternative}.  We will make this precise in the next subsection.

\begin{ex}  
	In addition to proving uniform exponential growth for a large class of cubical groups, \Cref{thm:mainthm} gives a single, unified proof that the following groups have uniform exponential growth.
\begin{enumerate}
	\item Non-elementary virtually torsion-free hyperbolic groups.  These groups are acylindrically hyperbolic \cite{Osin}, so we may apply \Cref{cor:ah}.  Uniform exponential growth was first shown for these groups by Gromov \cite{Gromov} and Delzant \cite{Delzant:hyperbolicSubgroups} (see \cite[Theorem(vii)]{GrigorchukDLHarpe:growth} for a precise statement) and generalized by Koubi \cite{Koubi} (without the torsion-free hypothesis).
	
	\item Non-exceptional mapping class groups.  These groups are acylindrically hyperbolic  \cite{MM:hyperbolicity, Bowditch} and virtually torsion-free  \cite[Corollary~1.5]{Ivanov}.  Uniform exponential growth was first shown by Anderson, Aramayona, and Shackleton \cite{AAS}. 
	
	
	\item 
	Many orientable 3-manifold groups. Specifically, if $M$ is geometric then it suffices that $M$ admits a complete metric locally isometric to $\mathbb{H}^3$ or $\mathbb{H}^2 \times \R$.  In the non-geometric case, it suffices to have $M$  be a flip graph $3$--manifold
	or certain mixed $3$--manifolds. 
	These groups are torsion-free and acylindrically hyperbolic  \cite{MinasyanOsin}, so we may apply \Cref{cor:ah}.  The class of hierarchically hyperbolic  3-manifold groups to which our theorem applies is broader than stated here, but rather technical. For example, the manifold needs not be prime, but cannot have any Nil or Sol components (see \cite[Remark~10.2]{BehrstockHagenSisto:HHS2}).  Uniform exponential growth is already known for 3--manifold groups (see for example \cite{diCerbo:3mfldUEG} and references therein).  
	For non-geometric 3--manifolds, this follows from the action on its JSJ--tree and work of Bucher and de la Harpe \cite{BucherDLHarpe}.   
	In the geometric case, this follows from work of Besson, Courtois, and Gallot \cite{BCG} for hyperbolic 3--manifolds and from the fact that  uniform exponential growth is inherited from quotients for Seifert fibered manifolds.  
	
	\item  Graph products of virtually torsion-free hierarchically hyperbolic groups.  Such groups are hierarchically hyperbolic by  \cite{BerlaiRobbio} and virtually torsion-free by \cite[Corollary~3.4]{JS}.  When the defining graph is not a join and $G \not\cong \Z_2 * \Z_2$, the space $\fontact S$ is non-elementary by work of Berlyne and Russell \cite{BerlyneRussel:GraphProd} extending work of Berlai and Robbio \cite{BerlaiRobbio}, and thus we may apply  \Cref{cor:CSnonelem}.
	This class includes free products and direct products of virtually torsion-free hierarchically hyperbolic groups.  Uniform exponential growth for graph products of this form is known to follow from work of Bucher and de la Harpe \cite{BucherDLHarpe}, as long as the underlying graph is not complete, and from Antol\'{i}n and Minasyan in the general case \cite[Corollary~1.5]{AM}. 

    \item A virtually torsion-free tree  of hierarchically hyperbolic groups satisfying the conditions of \cite[Corollary~8.24]{BehrstockHagenSisto:HHS2} or \cite{BerlaiRobbio}. For instance, groups of the form $G_1 \ast_C G_2$, where  $G_i$ is hyperbolic and $C$ is 2--ended, are hierarchically hyperbolic (\cite{RobbioSpriano:Hierarchical}). For the standard hierarchical structure on such groups, $\fontact S$ is a tree (which is not a quasi-line), and so we may apply \Cref{cor:CSnonelem}.  Uniform exponential growth  follows for such groups by Bucher and de la Harpe \cite{BucherDLHarpe}.
\end{enumerate}
\end{ex}

So far we have only provided   conditions that are sufficient to guarantee that an HHG is not quasi-isometric to a non-trivial product, whereas  \Cref{thm:mainthm} gives a more precise characterization of the product structure.  Thus,  \Cref{thm:mainthm} allows us to conclude that certain hierarchically hyperbolic groups which are quasi-isometric to a product still have uniform exponential growth. One example is the following.

\begin{ex}[Burger-Mozes] 
\label{ex:BM} 
Consider the group $G$ constructed by Burger and Mozes in \cite{BurgerMozes} as the first example of a torsion-free simple group which acts cocompactly on the product of two trees. It is known that \(G\) is quasi-isometric to the product of two trees (which are not lines).  Moreover,  $G$ was shown to be a hierarchically hyperbolic group by Behrstock, Hagen, and Sisto \cite[Section 8]{BehrstockHagenSisto:HHS1}.  However, there is no space $E$ such that $G$ is quasi-isometric to $\mathbb{Z}\times E$.  
Indeed, such a space $E$ would have to be a quasi-tree by work of Fujiwara and Whyte \cite[Theorem~0.1]{FujiwaraWhyte} together with bounds on the asymptotic (Assouad—Nagata) dimension \cite[Theorem~4.3]{DranishnikovSmith}, \cite[Theorem~2.4]{BrodskiyDydakLevinMitra}.
Such a quasi-isometry would induce a bi-Lipschitz homeomorphism on  asymptotic cones, contradicting a result of Kapovich and Leeb on the nonexistence of certain bi-Lipschitz maps from products of two trees \cite[Corollary~2.15]{KapovichLeeb}.  
By applying \Cref{thm:mainthm}, we obtain a new proof that $G$ has uniform exponential growth.  This result also follows the from the structure of $G$ as an amalgamated free product of two free groups and work of Bucher and de la Harpe \cite{BucherDLHarpe}.  \Cref{thm:main} in the appendix gives another  proof that $G$ has uniform exponential growth.

This example can be extended to give a new proof of uniform exponential growth for all 
\emph{BMW-groups} (this terminology is introduced and described in \cite{Caprace}).  This class of groups, which generalizes the original construction of Burger and Mozes, have uniform exponential growth.  A group $G$ is a BMW-group if it acts by isometries on the product of two trees $T_1 \times T_2$ such that every element preserves the product decomposition and the action on the vertex set of $T_1\times T_2$ is free and transitive. 
\end{ex}


\subsection{A quantitative Tits alternative}  

Most known proofs of uniform exponential growth, including the proof of \Cref{thm:mainthm}, demonstrate that  one can produce a pair of elements with bounded word length with respect to \emph{any} generating set that generate a free semigroup. In light of this, one can ask under what conditions one can find a pair of uniformly short elements which freely generate an actual subgroup.  In groups which satisfy a Tits alternative, producing a free basis with bounded word length can be seen as a \emph{quantitative} Tits alternative.

In our proof of \Cref{thm:technicalmainthm}, we use  
work of Breuillard and Fujiwara \cite{BreuillardFujiwara}
to produce short elements that generate a free semigroup.  Under the additional assumption of \emph{hierarchical acylindricity}, discussed in \Cref{subsec:applications}, we can upgrade our argument using earlier work of Fujiwara \cite{Fujiwara} to produce a genuine free subgroup, showing the following quantitative Tits alternative holds for hierarchically hyperbolic groups.

\begin{restatable*}{prop}{propQuanTits} \label{thm:qtits}  Let $(G,\s)$ be a virtually torsion-free hierarchically hyperbolic group such that \(G\) is not quasi-isometric to \(\mathbb{Z} \times E\) for any metric space \(E\). 
Suppose that either \begin{enumerate}
\item $\fontact S$ is non-elementary; or
\item  $G$ is hierarchically acylindrical.
\end{enumerate}
Then for any generating set $\genset$ of $G$, there exists a free subgroup of $G$ generated by two elements whose word length with respect to $\genset$ is uniformly bounded.
\end{restatable*}

\subsection{HHGs without uniform exponential growth}
We  now turn our attention to the class of hierarchically hyperbolic groups that do not have uniform exponential growth. Since every finitely generated abelian group is hierarchically hyperbolic, this provides a large class of examples that lack even (non-uniform) exponential growth. On the other hand, HHGs are finitely presented and satisfy a Tits alternative; that is, every finitely generated subgroup of a hierarchically hyperbolic group either contains a non-abelian free group or is virtually abelian
\cite{DHS:Corrigendum}. In light of this, we ask the following question.

\begin{question}\label{question: is there a third case?}
Does there exist a hierarchically hyperbolic group that is not virtually abelian and does not have uniform exponential growth?
\end{question}
Either a positive or negative answer to this question would be of significant interest.
A positive answer would prove that all hierarchically hyperbolic groups are either virtually abelian or have uniform exponential growth, while a negative answer would provide an example of a finitely presented group which has exponential but not uniform exponential growth, answering a question of Gromov. 
Although our techniques do not allow us to answer \Cref{question: is there a third case?}, we obtain a structural classification of the cases where uniform exponential growth does not (or may not) hold. 
We obtain rather restrictive conditions on the hierarchical structure a group must satisfy in order to answer \Cref{question: is there a third case?} in the affirmative.
\begin{theorem}\label{thm: the non ueg case}
Let \(G\) be a virtually torsion-free hierarchically hyperbolic group which is not virtually abelian and does not have uniform exponential growth. 
Then there exists a \(G\)--invariant set of pairwise orthogonal domains \(\bar{\mc{B}}\) such that for each \(U \in \bar{\mc{B}}\) the space \(\fontact U\) is uniformly a quasi-line, and for each \(V \not \in \bar{\mc{B}}\) either \(\fontact V\) is uniformly bounded, or \(V \bot U\) for all \(U \in \bar{\mc{B}}\). 
\end{theorem}

\subsection{About the proof of \Cref{thm:mainthm}} The proof of \Cref{thm:mainthm} has a similar structure to Mangahas's proof of uniform exponential growth for finitely generated subgroups of the mapping class group of a surface \cite{Mangahas}.  However, in this more general setting one needs to handle certain difficult behavior not present in the action of the mapping class group on the hierarchy of subsurface curve graphs. 
In particular, a general HHG does not contain a \emph{pure} subgroup (in the sense of Ivanov \cite{Ivanov}), that is, a finite index torsion-free subgroup such that for every domain $U$, 
elements that stabilize $U$ act on the space $\fontact U$ either loxodromically or trivially.
Indiscrete BMW-groups (see \Cref{ex:BM}) give one class of examples of such phenomena.  Indeed Caprace, Kropholler, Reid, and Wesolek \cite[Corollary~32(i), (iv)]{CKRW} show that in these groups every finite index subgroup contains infinite order elements which are non-trivial elliptic isometries with respect to the action on one of the tree factors.

The proof of \Cref{thm:mainthm} splits into two cases.  
In the first case, we assume that there exist short words that act loxodromically on the hyperbolic spaces associated to two non-orthogonal domains.
In this case we produce uniformly short 
powers that generate a free subgroup by playing ping-pong in the Cayley graph.  
If the first case doesn't hold, then we show that the action of (a finite index subgroup of) $G$ on the set of domains must fix a collection $\overline{\mathcal B}$ of pairwise orthogonal domains
pointwise.  
In this case, we show 
that either $\bar{\mc B}$ is a singleton or the top-level curve graph \(\fontact S\) has bounded diameter.  If $\bar{\mc B}$ is a singleton, we conclude that $G$ has uniform exponential growth by finding uniformly short elements of $G$ which are independent loxodromic isometries of $\fontact S$, and thus have short powers generating a free subgroup.   
If \(\fontact S\) has  bounded diameter, we conclude that $G$ is quasi-isometric to a product, and we next consider whether there are independent loxodromic isometries of $\fontact U$ for each $U\in\overline{\mathcal B}$.  If there are, then $G$ has uniform exponential growth.  Otherwise, we argue that each such $\fontact U$ is  quasi-isometric to a line and use this to give a more explicit description of the product structure of $G$. \\

\noindent{\bf Organization:}
In \Cref{sec:background} we review background material for uniform exponential growth, hierarchically hyperbolic groups, and tools to produce free (semi)groups.  
In \Cref{sec:quasilineStructure}, we give several structural results for when a hierarchically hyperbolic group contains invariant domains whose associated hyperbolic spaces are quasi-lines.  
This is followed by the proof of \Cref{thm:mainthm} in \Cref{sec:proof}, where we also prove all of the corollaries, \Cref{thm:qtits}, and \Cref{thm: the non ueg case}. 

In the appendix, Gupta and Petyt prove \Cref{thm:main}, a strengthening of \Cref{thm:mainthm} in the case of certain CAT(0) cubical groups which states that such groups either have uniform exponential growth or are virtually abelian.   Together,  \Cref{thm:mainthm,thm:main} prove \Cref{thm_main_and_appendix}.\\

\noindent{\bf Acknowledgments:}
The authors are grateful to David Hume for his advice on \Cref{ex:BM}.
The authors thank Jacob Russell, Mark Hagen, Jason Behrstock, and Sam Taylor for several helpful conversations and clarifications, and the anonymous referee for helpful comments.  TN and DS also give special thanks to their respective advisors, Dave Futer and Alessandro Sisto, for their ongoing support and their many helpful comments on early drafts of this paper.  We are also grateful to the organizers of the GAGTA 2018 conference where part of this collaboration began.

The first author was  partially supported by NSF Awards DMS--1803368 and DMS--2106906.  The second author was partially suppported by ISF grant \#660/20 and NSF grant DMS--1907708. The third author was partially supported by the Swiss National Science Foundation (grant \# 182186).

\section{Background and past results}
\label{sec:background}
		
		We begin by recalling some preliminary notions about metric spaces, maps between them, and group actions.  Given metric spaces $X,Y$, we use $\dist_{_X}, \dist_{_Y}$ to denote the distance functions in $X,Y$, respectively.   A map $f\colon X \to Y$ is:
		\begin{itemize}
		\item \emph{$K$--Lipschitz} if there exists a constant $K\geq 1$ such that for every $x,y\in X$, $\dist_{_Y}(f(x),f(y))\leq K \dist_{_X}(x,y)$;
		\item \emph{$(K,C)$--coarsely Lipschitz} if for every $x,y \in X$, $\dist_{_Y}(f(x),f(y))\leq K\dist_{_X}(x,y)+C$.			
		\item  a \emph{$(K,C)$--quasi-isometric embedding} if there exist constants $K\geq 1$ and $C\geq 0$ such that for all $x,y\in X$, \[\frac1K\dist_{_X}(x,y)-C\leq \dist_{_Y}(f(x),f(y))\leq K\dist_{_X}(x,y)+C,\]  
		\item  a \emph{$(K,C)$--quasi-isometry} if it is a $(K,C)$--quasi-isometric embedding and, \emph{coarsely surjective}, that is, $Y$ is contained in the $C$--neighborhood of $f(X)$.  In this case, we say $X$ and $Y$ are \emph{quasi-isometric}.  
		\end{itemize}
For any interval $I\subseteq \mathbb R$, the image of an isometric embedding $I\to X$ is a \emph{geodesic} and the image of a $(K,C)$--quasi-isometric embedding $I\to X$ is a \emph{$(K,C)$--quasigeodesic}.  A space $X$ is a \emph{quasi-line} if it is quasi-isometric to $\R$.

If any two points in $X$ can be connected by a $(K,C)$--quasigeodesic, then we say $X$ is a \emph{$(K,C)$--quasigeodesic space}.  If $K=C$, we may simply say that $X$ is a \emph{$K$--quasigeodesic space}.   For all of the above notions, if the particular constants $K,C$ are not important, we may drop them and simply say, for example, that a map is a quasi-isometry.

Throughout this paper, we will assume that all group actions are by isometries.  The action of a group $G$ on a metric space $X$
is \emph{proper} if the set $\{g\in G\mid gB\cap B\neq \emptyset\}$ is finite for every bounded subset $B\subseteq X$.  The action is  \emph{cobounded} (respectively, \emph{cocompact}) if there exists a bounded (respectively, compact) subset $B\subseteq X$ such that $X=\bigcup_{g\in G} gB$.  If a group $G$ acts on metric spaces $X$ and $Y$, we say a map $f\colon X\to Y$ is \emph{$G$--equivariant} if for every $x\in X$,  $f(gx)=gf(x)$.

Given a metric space $X$ and a subspace $Y$, we define the \emph{$A$--neighborhood of $Y$} to be  
$$N_A(Y)=\{x\in X \mid \dist_{_X}(x,Y)\leq A\}.$$
Let $X$ be a geodesic metric space and let $x,y,z\in X$.  We denote by $[x,y]$ a geodesic segment between $x$ and $y$.  A geodesic triangle with vertices $x,y,z$ is \emph{$\delta$--slim} if there is a constant $\delta\geq 0$ such that for any point  $p\in[x,y]$, there is a point $m\in[y,z]\cup[x,z]$ such that $\dist_{_X}(p,m)\leq \delta$.  The space $X$ is \emph{$\delta$--hyperbolic} if there is a constant $\delta\geq 0$ such that every geodesic triangle is $\delta$--slim.

	\subsection{Uniform exponential growth}

	Given a finite collection of elements $X$ containing the identity in a group, the \emph{growth function} of $X$ is 
$$ \beta_X(n) = \abs{X^n}, $$ 
	where $X^n=\{x_1\dots x_n\mid x_i\in X\}$.
	This function $\beta_X(n)$ counts the number of elements that can be expressed
	as words in the alphabet $X$ with length at most $n$.
	The \emph{exponential growth rate} of a finite subset $X$ of a group is 
$$ \lambda(X) := \lim_{n \ra \infty}\frac{\log(\beta_X(n))}{n}. $$

\begin{defn}[(Uniform) Exponential growth]
	A finitely generated group is said to have \emph{exponential growth} if there is a finite generating set $X$ such that 
	$$ \lambda(X) > 0. $$
	Such a group has \emph{uniform exponential growth} if the infimum over all finite generating sets is bounded away from 0, that is,  
	$$ \lambda_0 \defeq \inf\limits_{\stackrel{\abrackets{X} = G}{\abs{X} < \infty}}\lambda(X) > 0. $$
\end{defn}

\begin{remark}
	One can also use the function 	$$ \om(X) := \lim_{n \ra \infty}\sqrt[n]{\beta_X(n)} $$ in place of $\lambda(X)$ to give an equivalent characterization of exponential growth rate.  In this case, the growth is uniform if it is uniformly bounded above 1.
\end{remark}

	If $G = F_2$ is a  free group of rank two and $X = \braces{1, a, b}$ is a generating set, then it is easy to see that $\abs{X^n} \geq 2^n$.  Hence, $\lambda(X) \geq \log(2)$.  In fact, since any generating set contains a pair of noncommuting elements and nonabelian subgroups of a free group are free, we have $\lambda(X') \geq \log(2)$ for any generating set $X'$.  Therefore, $\lambda_0 \geq \log 2>0$, and so $F_2$ has uniform exponential growth.  By the same reasoning, free semigroups have uniform exponential growth.

In light of this, most known proofs of uniform exponential growth make use of the following observation.

\begin{obs}[Short free semigroups witness uniform exponential growth]
	\label{shortFreeImpliesUEG}
	If there is a constant $N$ depending only on the group $G$ such that for any generating set $X$ there exists two elements with $X$--length at most $N$ whose positive words generate a free semigroup, then $G$ has uniform exponential growth with $\lambda_0 \geq \frac{\log(2)}{N}$. 
\end{obs}

The following result of Shalen and Wagreich gives bounds on the growth of a group given the growth of a finite index subgroup.  

\begin{lemma}[{{\cite[Lemma 3.4]{ShalenWagreich}}}]
	\label{finiteIndexWordLength}
	Let $G$ be a group with finite generating set $X$, and let $H$ be a finite index subgroup with $[G:H] = d$. Then there exists a generating set for $H$ all of whose elements of have $X$--length at most $2d - 1$. 
\end{lemma}
This implies that if $[G: H] = d$ then 
$$ \lambda_0(G) \geq \frac{1}{2d - 1}\lambda_0(H), $$
thus, uniform exponential growth passes to finite index supergroups.

	\subsection{Finding free (semi)groups}
In this section, we give multiple ways to find free (semi)groups given an  action of a group on a hyperbolic metric space. We will assume all actions on metric spaces are by isometries.  Together with \Cref{shortFreeImpliesUEG}, these will be our key tools to show uniform exponential growth.

The first is a version of the standard ping-pong lemma. 
\begin{lemma}
	Let $G$ be a group acting on a set $X$, and let $a,b\in G$ of infinite order.  Suppose there exist disjoint non-empty subsets $X_1,X_2\subseteq X$ such that $a^{n}. X_2\subseteq X_1$ and $b^{n}. X_1\subseteq X_2$ for all \(n \neq 0\).  Then $\langle a,b\rangle$ is a free group of rank 2.
\end{lemma}

Let $G$ be a group acting on a hyperbolic metric space $X$ with basepoint $x_0\in X$, and let $g\in G$.  The \emph{(stable) translation length of $g$} is defined to be $\tau(g)=\lim_{n\to\infty}\frac{d(x_0,g^nx_0)}{n}$.  If $\tau(g)>0$, then $g$ is a \emph{loxodromic} isometry of $X$.  Equivalently, $g$ is loxodromic if it fixes exactly two points in the boundary $\partial X$ of $X$.  Such isometries act as translation along a quasi-geodesic \emph{axis} in $X$.  Two loxodromic isometries are \emph{independent} if their fixed point sets in $\partial X$ are disjoint. 

The following result gives a method for producing free semigroups from an action on a hyperbolic space.  While the statement is likely well-known, Breuillard and Fujiwara give an explicit proof in this context  \cite{BreuillardFujiwara}.  Their proof generalizes  the analogous result for simplicial trees due to Bucher and de la Harpe \cite{BucherDLHarpe}.  When the hyperbolic space is a  Hadamard manifolds with $K\leq -1$, the result is due to Besson, Courtois, and Gallot \cite{BCG}. 

\begin{prop}
[{\cite[Proposition~11.1]{BreuillardFujiwara}}]
\label{BreuillardFujiwara:freeSemigroup}
	For $\delta \geq 0$ let $X$ be a $\delta$--hyperbolic space, and $g,h \in \Isom(X)$.  Suppose $g$ and $h$ are loxodromic isometries whose fixed point sets in $\partial X$ are not equal and $\tau(g), \tau(h) > 10000\delta$. 
	Then some pair in  $\braces{g^{\pm 1}, h^{\pm 1}}$ generates a free semigroup.
\end{prop}

In particular, this result shows that given a pair of elements with stable translation length bounded from below, there are powers depending only on the displacement bound that generate a free semigroup.  While it would be sufficient to use \Cref{BreuillardFujiwara:freeSemigroup} to 
show uniform exponential growth, under the additional assumption that the action is acylindrical, we can construct genuine free subgroups generated by short conjugates of a single loxodromic.  

\begin{theorem}
[{\cite[Proposition~2.3(2)]{Fujiwara}}] \label{Fujiwara:freeSubgroup}
	If $G$ acts acylindrically on a $\delta$--hyperbolic space containing elements $a,b \in G$ such that $a$ acts loxodromically and $ba^nb^{-1} \neq a^{\pm n}$ for any $n \neq 0$ then there is a constant power $p$ depending on $\delta$ and the acylindricity constants such that $\abrackets{a^k, ba^kb^{-1}} = \F_2 < G$ for all $k \geq p$.
\end{theorem}

We note that the requirement that $ba^nb^{-1}\neq a^{\pm n}$ for any $n$ ensures that $a^k$ and $ba^kb^{-1}$ are independent loxodromic isometries.

	\subsection{Definition of a hierarchically hyperbolic group}

We begin this subsection by recalling the definition of a hierarchically hyperbolic space as 
given in \cite{BehrstockHagenSisto:HHS2}.

\begin{defn}[Hierarchically hyperbolic space]\label{defn:HHS}
The quasigeodesic space  $(\cuco X,\dist_{_\cuco X})$ is a \emph{hierarchically hyperbolic space (HHS)} if there exists $\delta\geq0$, an index set $\mathfrak S$, and a set $\{\fontact W:W\in\mathfrak S\}$ of $\delta$--hyperbolic spaces $(\fontact W,\dist_{_W})$,  such that the following conditions are satisfied:
\begin{enumerate}

\item\textbf{(Projections.)}\label[axiom]{item:dfs_curve_complexes} 
There is
a set $\{\pi_W\colon \cuco X\rightarrow2^{\fontact W}\mid W\in\mathfrak S\}$
of \emph{projections} sending points in $\cuco X$ to sets of diameter
bounded by some $\xi\geq0$ in the various $\fontact W\in\mathfrak S$.
Moreover, there exists $K$ so that each $\pi_W$ is $(K,K)$--coarsely
Lipschitz and $\pi_W(\cuco X)$ is $K$--quasiconvex in $\fontact
W$.

 \item \textbf{(Nesting.)} \label[axiom]{item:dfs_nesting} 
 $\mathfrak S$ is
 equipped with a partial order $\nest$, and either $\mathfrak
 S=\emptyset$ or $\mathfrak S$ contains a unique $\nest$--maximal
 element; when $V\nest W$, we say $V$ is \emph{nested} in $W$.  (We
 emphasize that $W\nest W$ for all $W\in\mathfrak S$.)  For each
 $W\in\mathfrak S$, we denote by $\mathfrak S_W$ the set of
 $V\in\mathfrak S$ such that $V\nest W$.  Moreover, for all $V,W\in\mathfrak S$
 with $V$ properly nested in $W$ there is a specified subset
 $\rho^V_W\subset\fontact W$ with $\diam_{\fontact W}(\rho^V_W)\leq\xi$.
 There is also a \emph{projection} $\rho^W_V\colon \fontact
 W\rightarrow 2^{\fontact V}$. 
 
 \item \textbf{(Orthogonality.)} 
 \label[axiom]{item:dfs_orthogonal} 
 $\mathfrak S$ has a symmetric and
 anti-reflexive relation called \emph{orthogonality}: we write $V\orth
 W$ when $V,W$ are orthogonal.  Also, whenever $V\nest W$ and $W\orth
 U$, we require that $V\orth U$.  We require that for each
 $T\in\mathfrak S$ and each $U\in\mathfrak S_T$ for which
 $\{V\in\mathfrak S_T\mid V\orth U\}\neq\emptyset$, there exists $W\in
 \mathfrak S_T-\{T\}$, so that whenever $V\orth U$ and $V\nest T$, we
 have $V\nest W$.  Finally, if $V\orth W$, then $V,W$ are not
 $\nest$--comparable.
 
 \item \textbf{(Transversality and consistency.)}
 \label[axiom]{item:dfs_transversal} 
 If $V,W\in\mathfrak S$ are not
 orthogonal and neither is nested in the other, then we say $V,W$ are
 \emph{transverse}, denoted $V\trans W$.  There exists
 $\kappa_0\geq 0$ such that if $V\trans W$, then there are
  sets $\rho^V_W\subseteq\fontact W$ and
 $\rho^W_V\subseteq\fontact V$ each of diameter at most $\xi$ and 
 satisfying: $$\min\left\{\dist_{_
 W}(\pi_W(x),\rho^V_W),\dist_{_
 V}(\pi_V(x),\rho^W_V)\right\}\leq\kappa_0$$ for all $x\in\cuco X$.
 
 For $V,W\in\mathfrak S$ satisfying $V\nest W$ and for all
 $x\in\cuco X$, we have: $$\min\left\{\dist_{_W}(\pi_W(x),\rho^V_W),\diam_{\fontact
 V}(\pi_V(x)\cup\rho^W_V(\pi_W(x)))\right\}\leq\kappa_0.$$ 
 
 The preceding two inequalities are the \emph{consistency inequalities} for points in $\cuco X$.
 
 Finally, if $U\nest V$, then $\dist_{_W}(\rho^U_W,\rho^V_W)\leq\kappa_0$ whenever $W\in\mathfrak S$ satisfies either that $V$ is properly nested in \( W\) or that $V\trans W$ and $W\not\perp U$. 
 
 \item \textbf{(Finite complexity.)} \label[axiom]{item:dfs_complexity} 
 There exists $n\geq0$, the \emph{complexity} of $\cuco X$ (with respect to $\mathfrak S$), so that any set of pairwise--$\nest$--comparable elements has cardinality at most $n$.
  
 \item \textbf{(Large links.)} \label[axiom]{item:dfs_large_link_lemma} 
 Thereexist $\lambda\geq1$ and $E\geq\max\{\xi,\kappa_0\}$ such that the following holds.
Let $W\in\mathfrak S$ and let $x,x'\in\cuco X$.  Let
$N=\lambda\dist_{_W}(\pi_W(x),\pi_W(x'))+\lambda$.  Then there exists $\{T_i\}_{i=1,\dots,\lfloor
N\rfloor}\subseteq\mathfrak S_W-\{W\}$ such that for all $T\in\mathfrak
S_W-\{W\}$, either $T\in\mathfrak S_{T_i}$ for some $i$, or $\dist_{_
T}(\pi_T(x),\pi_T(x'))<E$.  Also, $\dist_{_
W}(\pi_W(x),\rho^{T_i}_W)\leq N$ for each $i$.

 \item \textbf{(Bounded geodesic image.)} \label[axiom]{item:dfs:bounded_geodesic_image} 
 There exists $E>0$ such that 
 for all $W\in\mathfrak S$,
 all $V\in\mathfrak S_W-\{W\}$, and all geodesics $\gamma$ of
 $\fontact W$, either $\diam_{\fontact V}(\rho^W_V(\gamma))\leq E$ or
 $\gamma\cap N_E(\rho^V_W)\neq\emptyset$. 
  
 \item \textbf{(Partial Realization.)} \label[axiom]{item:dfs_partial_realization} 
 There exists a constant $\alpha$ with the following property. Let $\{V_j\}$ be a family of pairwise orthogonal elements of $\mathfrak S$, and let $p_j\in \pi_{V_j}(\cuco X)\subseteq \fontact V_j$. Then there exists $x\in \cuco X$ so that:
 \begin{itemize}
 \item $\dist_{_{V_j}}(x,p_j)\leq \alpha$ for all $j$,
 \item for each $j$ and 
 each $V\in\mathfrak S$ with $V_j\nest V$, we have 
 $\dist_{_V}(x,\rho^{V_j}_V)\leq\alpha$, and
 \item if $W\trans V_j$ for some $j$, then $\dist_{_W}(x,\rho^{V_j}_W)\leq\alpha$.
 \end{itemize}

\item\textbf{(Uniqueness.)} \label[axiom]{item:dfs_uniqueness}
For each $\kappa\geq 0$, there exists
$\theta_u=\theta_u(\kappa)$ such that if $x,y\in\cuco X$ and
$\dist_{_\cuco X}(x,y)\geq\theta_u$, then there exists $V\in\mathfrak S$ such
that $\dist_{_V}(x,y)\geq \kappa$.
\end{enumerate}

\end{defn}

For ease of readability, given $U\in\mathfrak S$, we typically 
suppress the projection map $\pi_U$ when writing distances in
$\fontact U$, that is, given $x,y\in\cuco X$ and $p\in\fontact U$ we
write $\dist_{_U}(x,y)$ for $\dist_{_U}(\pi_U(x),\pi_U(y))$ and
$\dist_{_U}(x,p)$ for $\dist_{_U}(\pi_U(x),p)$. 

Heuristically, a hierarchically hyperbolic structure on a space $\X$ is a means of organizing the space by the coarse geometry of the product regions in $\X$ and their interactions.  Nesting gives a notion of sub-product regions and subspaces.  Transversality gives a notion of separate or isolated subspaces.  Orthogonality gives a notion of independent subspaces that together span a product region in $\X$.

An important consequence of being an HHS is the existence of a distance formula, which relates distances in $\mathcal X$ to distances in the hyperbolic spaces $\fontact U$. 
The notation $\ignore{x}{s}$ means include $x$ in the 
sum if and 
only if $x>s$.

\begin{theorem}[Distance formula;
    {\cite[Theorem~4.5]{BehrstockHagenSisto:HHS2}}\label{thm:distance_formula}]\label{thm:distformula}
Let $(\cuco X, \mathfrak S)$ be a hierarchically hyperbolic space.  Then
there exists $s_0$ such that for all $s\geq s_0$, there exist $C,K$ so
that for all $x,y\in\cuco X$,
$$\dist(x,y) \underset{\text{\scriptsize $K,C$}}{\asymp} \sum_{U\in\mathfrak
S}\ignore{\dist_{_U}(x,y)}{s}.$$
\end{theorem}

We will now define the main object of this paper, hierarchically hyperbolic groups (HHG). Intuitively, a hierarchically hyperbolic group is a group whose Cayley graph is an HHS such that the action of the group on its Cayley graph is compatible with the HHS structure. The compatibility of the action is a key requirement: it is tedious but straightforward to verify that the definition of HHS is quasi-isometry invariant, whereas it is unknown if being an HHG is preserved under quasi-isometry. 
We first recall the definition of a hieromorphism.

\begin{defn}[Hieromorphism; {\cite[Definition 1.20]{BehrstockHagenSisto:HHS2}}]
	\label{item:dfs_hieromorphism}
A \emph{hieromorphism} between the hierarchically hyperbolic spaces \((\X, \mf{S})\) and \((\X', \mf{S}')\) consists of a map \(f\colon \X \to \X'\), an injection \(f^\diamond \colon \mf{S} \to \mf{S}'\) and a collection of quasi-isometric embeddings \(f^\ast_U \colon \fontact U \to \fontact f^\diamond(U)\) such that the two following diagrams uniformly coarsely commute (whenever defined).
\begin{center}
\begin{tikzcd}
\mc{X} \arrow{r}{f}\arrow{d}{\pi_U}     &     \mc{X}' \arrow{d}{\pi_{f^\diamond (U)}}   & &  \fontact U \arrow{r}{f^{\ast}_U}\arrow{d}{\rho^U_V} & \fontact f^\diamond (U)\arrow{d}{\rho_{f^\diamond V}^{f^\diamond U}} \\
\fontact U  \arrow{r}{f^{\ast}_U} & \fontact f^\diamond (U) & & \fontact V\arrow{r}{f^{\ast}_V} & \fontact f^\diamond(V)
\end{tikzcd}\end{center}
\end{defn}
As the functions $f, 
f^*(U),$ and $f^\diamond$ all have distinct domains, it is 
often clear from the 
context which is the relevant map; in that case we periodically abuse 
notation slightly by dropping the superscripts and simply calling all of the maps $f$.

Note that the definition does not have any requirement on the map \(f\). This is because the distance formula (\Cref{thm:distformula}) implies that \(f\) is determined up to uniformly bounded error by the map \(f^\diamond\) and the collection \(\{f^{\ast}_U \mid U \in \mf{S}\}\). 
The fact that a hieromorphism is coarsely determined by its action on the hierarchical structure will play a key role in the definition of a hierarchically hyperbolic group.  
\begin{defn} Let \((\mc{X},\mf{S})\) be a hierarchically hyperbolic space. 
An \emph{automorphism} of \((\mc{X}, \mf{S})\) is a hieromorphism \(f \colon \mc{X} \to \mc{X}\), such that the map \(f^\diamond \colon \mf{S} \to \mf{S}\) is a bijection,  and the maps \(f^{\ast}_U \colon \fontact U \to \fontact f^\diamond (U)\) are isometries.
Two automorphisms \(f, f'\) are \emph{equivalent} if \(f^\diamond = (f')^\diamond\) and \(\phi_U = \phi_U'\) for all \(U\). Given \(f\), we define a quasi-inverse \(\bar{f}\) by setting \(\bar{f}^\diamond = (f^\diamond)^{-1}\) and \(\bar{\phi}_{f^\diamond (U)} = \phi_U^{-1}\) (then \(\bar{f}\) is determined by the distance formula). The set of such equivalence classes forms a group, denoted \(\mathrm{Aut}(\mf{S})\).
\end{defn}
\begin{defn}
A group \(G\) is said to be \emph{hierarchically hyperbolic} if there is a hierarchically hyperbolic space \((\mc{X}, \mf{S})\) and an action \(G \to \mathrm{Aut}(\mf{S})\) such that the quasi-action of \(G\) on \(\mc{X}\) is geometric and \(\mf{S}\) contains finitely many \(G\)--orbits. 
\end{defn}

\begin{remark}
	For any hierarchically hyperbolic group $\mc{X}$ can be taken to be the Cayley graph of $G$ with respect to any finite generating set.  In this case, $G$ acts on $\mc{X}$ by isometries.  We adopt this convention for the remainder of the paper and use the notation $(G, \s)$ to denote this structure. 
\end{remark}

\begin{remark} \label{rem:uniformbound}
By the definition of a hierarchically hyperbolic group, there is finite set of domains $U_1,\dots, U_k$ such that for every $W\in\s$, there is some $i=1,\dots, k$ such that $\fontact W$ is isometric to $\fontact U_i$.  It follows that for every $W\in\s$, the diameter of $\fontact W$ is either infinite or \emph{uniformly} bounded. 
\end{remark}

\begin{remark}\label{rem:rhocommutes}
Durham, Hagen, and Sisto show in \cite[Section~2.1]{DHS:Corrigendum} that for any hierarchically hyperbolic group, we may assume without loss of generality that the diagrams  in Definition \ref{item:dfs_hieromorphism} genuinely commutes, rather than only coarsely commute.  That is, given any $U\in \s$, any $V\in \s$ so that $\rho^U_V$ is defined, and any $g,h\in G$, we have 
\[
g\pi_U(h)=\pi_{gU}(gh) \qquad \textrm{and} \qquad g\rho^U_V=\rho^{gU}_{gV}.
\]
\end{remark}

In what follows we will consider an HHG \((G, \mf{S})\) with respect to different finite generating sets. Let \(X\) and \(Y\) be two finite generating sets for a group \(G\), and suppose that an HHG structure \((G, \mf{S})\) is given, where distances in \(G\) are measured with \(\dist_{_X}\). Then the identity provides an equivariant quasi-isometry between \((G, \dist_{_X})\) and \((G, \dist_{_Y})\). Note that this provides a hierarchically hyperbolic group structure on \((G, \dist_{_Y})\), where all the constants of the hierarchy axioms are the same, except the ones that involve distances in \(G\). In particular, the only two such constants are the \(K\) of the projections of  \Cref{item:dfs_curve_complexes}, and the constant \(\theta_u\) of \Cref{item:dfs_uniqueness}.

\begin{remark}  
\label{rem:hierarchyConstant}	
	We say a constant $k$ depends only on $(G,\s)$ when $k$ depends only on the constants in the definition of the hierarchically hyperbolic structure on $G$ which are independent of the generating set.  Further, we will frequently refer to $D = \max\braces{ \delta, \xi, \kappa_0, n, E }$ as the \emph{hierarchy constant}, which is also independent of the generating set.   
\end{remark}

\begin{lemma} \label{lem:closerhoimpliestrans}
	Let $U,W,V\in\s$ be such that $U$ and $W$ properly nest into $V$.  If $\dist_{_V}(\rho^U_V,\rho^W_V)>2D$, then $U\trans W$.
\end{lemma}

\begin{proof}
	If $U\nest W$ or $W\nest U$, then $\dist_{_V}(\rho^U_V,\rho^W_V)\leq D$ by the transversality and consistency axiom, which contradicts our assumption.  If $U\perp W$, then there is a partial realization point $x\in\X$ such that $\dist_{_V}(x,\rho^U_V)\leq D$ and $\dist_{_V}(x,\rho^W_V)\leq E$.  It follows that $\dist_{_V}(\rho^U_V,\rho^W_V)\leq 2D$, which contradicts our assumption.  Therefore $U\trans W$.
\end{proof}

\begin{defn}[Normalized HHS]
The HHS \((\X, \mf{S})\) is \emph{normalized} if there exists \(C\) such that for each \(U \in \mf{S}\) one has \(\fontact U = N_C (\pi_U(\X))\).
\end{defn}

\noindent \textbf{Standing assumption.} By \cite[Proposition~1.16]{DurhamHagenSisto:Boundaries}, we can and will assume that all hierarchically hyperbolic spaces are normalized.

	\subsection{Preliminaries on hierarchically hyperbolic groups}
In this section, we recall the classification of hierarchical automorphisms from \cite{DurhamHagenSisto:Boundaries} and related results.

\begin{defn}[Big set]
	The \emph{big set} of an element is the collection of all domains onto whose associated hyperbolic spaces the orbit map is unbounded, that is, for an element $g \in \Aut(\s)$ and base point $x \in \X$ the \emph{big set} is
	$$ \bigset{g} = \braces{U \in \s \; \middle |\; \diam_{\fontact U}\pns{\abrackets{g}.x} \text{ is unbounded} }. $$
\end{defn}
Note that this collection is independent of base point.
\begin{remark} The elements of $\bigset{g}$ must all be pairwise orthogonal.  It follows immediately that $|\bigset{g}|$ is uniformly bounded by the constant from  \Cref{item:dfs_complexity} of  \Cref{defn:HHS}.  For the rest of the paper, we denote this number by $N$.
\end{remark}

\begin{defn}  
An automorphism of a hierarchically hyperbolic space is \emph{elliptic} if it acts with bounded orbits on $\X$.  It is \emph{axial} if its orbit map induces a quasi-isometric embedding of a line in $\X$.  
\end{defn}

\begin{prop}[{\cite[Lemma 6.3, Proposition 6.4, \& Theorem 7.1]{DurhamHagenSisto:Boundaries}}]
	\label{prop:HHGisomClassification}
Let \((G, \mf{S})\) be a hierarchically hyperbolic group. Then there exists \(M = M(\mf{S})\) between $0$ and $N!$ so that for all \(g\in \operatorname{Aut}(\mf{S})\) the following hold. 
\begin{enumerate}
\item \(g\) is either elliptic or axial;
\item \(g\) is elliptic if and only if \(\bigset{g} = \emptyset\);
\item for every \(U \in \bigset{g}\), we have \({(g^\diamond)}^M (U) = U\).
\end{enumerate}
\end{prop}        

\begin{remark} \label{rmk:nonemptybigset} An element $g\in G$ is finite order if and only if $\bigset{g}=\emptyset$ \cite[Lemma~1.7]{AB}.  Therefore, if $G$ is a torsion-free HHG, then every element of $G$ has a non-empty big set.
\end{remark}
    
Given an infinite order element $g\in G$ and a domain $U\in \s$ for which $g$ is loxodromic with respect to the action on $\fontact U$, we let $\tau_U(g)$
denote the (stable) translation length of $g$ in this action.

\begin{lemma}
[{\cite[Lemma~1.8]{AB}}]
\label{lem:translengthbound} 
	Let $(G,\s)$ be a hierarchically hyperbolic group.  There exists a constant $\displacementBound > 0$ such that, for every infinite order element $g\in G$ and every $U\in\B(g)$, we have $\displacement{U}{g}\geq \displacementBound$. 
\end{lemma}

Throughout the paper, it will be important for us to pass to certain finite index subgroups while maintaining the hierarchical structure of the group. We do this with the following lemma.  

\begin{lemma}\label{lem: passing to f.i.subgroup}
	Let $(G,\s)$ be a hierarchically hyperbolic group, and let $H$ be a finite index subgroup of $G$.  Then $(H,\s)$ is a hierarchically hyperbolic group with the same hierarchical structure as \(G\). Moreover, the property of being normalized
is preserved under passing to finite-index subgroups.
\end{lemma}

\begin{proof}
	Since \(G\) is an HHG, we have an embedding \(G \hookrightarrow \Aut(\mf{S})\), and hence an embedding \(H \hookrightarrow \Aut(\mf{S})\).	 Since \(H\) is of finite index in \(G\), we have that \(H\) still acts on \(\mf{S}\) with finitely many orbits. Moreover, since \(H\) coarsely coincides with \(G\), the uniform quasi-action of \(H\) on \(G\)	is metrically proper and cobounded. This proves that \(H\) is an HHG. 
	
	Suppose that \((G, \s)\) is normalized. 
	For each \(U \in \mf{S}\), the map \(\pi_U \colon H \to \fontact U\) is defined as the restriction of \(\pi_U \colon G \to \fontact U\). Since the latter is coarsely surjective by hypothesis, and since \(H\) coarsely coincides with \(G\), we obtain that \(\pi_U \colon H \to \fontact U\) is coarsely surjective, yielding that \(H\) is normalized. 
\end{proof}

\subsection{Hierarchical structures as coordinate systems} \label{sec:coordsys}
In this section, we will describe a product decomposition of $G$.  We begin by recalling the definition of a $\kappa$--consistent tuple.

\begin{defn}[{\cite[Definition~1.17]{BehrstockHagenSisto:HHS2}}]
\label{defn: consistent tuple}

Fix \(\kappa \geq 0\), and let \(\vec{b} \in \PiX\) be a tuple such that for each \(W \in \mf{S}\), the coordinate \(b_W\) is a subset of \(\fontact W\) of diameter at most \(\kappa\). The tuple \(\vec{b}\) is \emph{\(\kappa\)--consistent} if:
\begin{enumerate}
\item 
    \(\dist_{_W} (b_W, \pi_W (\mc{X})) \leq \kappa\) for all \(W \in \mf{S}\);
\item 
    \(\mathrm{min} \left\{\dist_{_W}(b_W, \rho_W^V), \dist_{_V}(b_V, \rho^W_V)  \right\} \leq \kappa \), whenever \(V \pitchfork W\);
\item 
    \(\mathrm{min} \left\{ \dist_{_W}(b_W, \rho_W^V, \mathrm{diam}_{\fontact V} (b_V \cup \rho_V^W (b_W))   \right\} \leq \kappa \), whenever \(V \sqsubseteq W\).
\end{enumerate}
We denote the subset of \(\PiX\) consisting of \(\kappa\)--consistent tuples by \(\Omega_\kappa\). 
\end{defn}
\begin{remark}
Note that for \(\kappa\) large enough, the first condition holds automatically if \((\mc{X}, \mf{S})\) is a normalized HHS.
\end{remark}

The goal of this section is to prove a sufficient condition on the index set \(\mf{S}\) under which the group, \(G\), quasi-isometrically decomposes as a product. This result can be deduced from discussions in \cite[Sections~3~\&~5]{BehrstockHagenSisto:HHS2}; we restate it here, along with its justification, for the sake of clarity and completeness.

\begin{prop} \label{prop:gprod}
Let \((\X, \mf{S})\) be a hierarchically hyperbolic space and let \(\bar{\mf{S}}\) consist of all \(W \in \mf{S}\) such that \(\fontact W\) has infinite diameter. Suppose that \(\bar{\mf{S}}\) can be partitioned as \(\mf{T}_1 \sqcup \cdots \sqcup \mf{T}_n\) where \(\mf{T}_i\neq \emptyset\) for all \(i\) and every element of \(\mf{T}_i\) is orthogonal to every element of \(\mf{T}_j\) for \(i \neq j\). Then there are infinite diameter metric spaces \(Y_i\) such that \(\X\) is quasi-isometric to \(Y_1 \times \cdots \times Y_n\).  Moreover, each $Y_i$ can be equipped with an HHS structure.
\end{prop}

The main technical ingredient to prove the proposition is to establish a connection between \(G\) and the set of consistent tuples. 
First, note that there is a map \(\pi \colon \mc{X} \to \PiX\) defined by associating to each \(x \in \mc{X}\) the tuple \(\{\pi_W(x)\}_{W \in \mf{S}}\).
The standing assumption that \((\mc{X}, \mf{S})\) is normalized yields a constant \(C\) such that all projections \(\pi_W\) are \(C\)--coarsely surjective. Thus, by setting \(\kappa_1 = \mathrm{max} \{C, \kappa_0, \xi\}\), \Cref{item:dfs_curve_complexes,item:dfs_transversal} of \Cref{defn:HHS} give that for each \(\kappa \geq \kappa_1\), the map \(\pi\) has image in \(\Omega_\kappa\). The following theorem should be thought of as saying that the projection \(\pi\) has a quasi-inverse. 
\begin{theorem}[{\cite[Theorem 3.1]{BehrstockHagenSisto:HHS2}}]\label{thm: realization theorem}
For each \(\kappa \geq 1\) there exist \(\theta_e, \theta_u \geq 0\) such that the following holds. Let \(\vec{b} \in \Omega_\kappa\) be a \(\kappa\)--consistent tuple, and for each \(W\) let \(b_W\) denote the \(\fontact W\)--coordinate of \(\vec{b}\).  Then the set \(\Psi(\vec b) \subseteq \mc{X}\) defined as all \(x \in \mc{X}\)  so that \(\dist_{_W} (b_W, \pi_W(x)) \leq \theta_e\) for all \(\fontact W \in \mf{S}\) is non empty and has diameter at most \(\theta_u\). 
\end{theorem}
 The reason why "\(\Psi\) is a quasi-inverse of \(\pi\)" is not a precise statement is that we did not equip \(\Omega_\kappa\) with a metric. The distance formula (\Cref{thm:distance_formula}) gives a constant \(s_0\) such that for each \(s \geq s_0\) there is a map \(f_s \colon \Omega_\kappa \times \Omega_\kappa\to \mathbb{R}\) defined as 
 \[
 	f_s(\vec a , \vec b ) = \sum_{W \in \mf{S}} \ignore{ \dist_{_W}(a_W, b_W) }{s},
 \]
 such that for every \(x,y \in \mc{X}\), the quantities \(f_s(\pi(x), \pi(y))\) and \(\dist_{_\cuco{X}} (x,y)\) are comparable. However, note that the map \(f_s\) is not a distance: it does not satisfy the triangle inequality and there exists \(\vec a \neq \vec b\) such that \(f_s (\vec a , \vec b) = 0\). To remedy this, we equip \(\Omega_\kappa\) with the subspace metric coming from \(\Psi\), which we denote by \(\dist_{_\cuco{X}}\) with an abuse of notation. 
 
The next ingredient in the proof of \Cref{prop:gprod} is to show that one needs only focus on domains whose associated hyperbolic spaces have sufficiently large diameter. We first concern ourselves with subdividing \(\mf{S}\) into blocks. Let \(\mf{S}'\subseteq \mf{S}\) be any subset. It is straightforward to see that concept of a consistent tuple (\Cref{defn: consistent tuple}) can be generalized to \(\prod_{W \in \mf{S}'}2^{\fontact W}\). Let \(\Omega_\kappa^{\mf{S}'}\) be the set of \(\kappa\)--consistent tuples of \(\prod_{W \in \mf{S}'}2^{\fontact W}\).   
 
\begin{defn}
Let \((\mc{X}, \mf{S})\) be a hierarchically hyperbolic space and suppose that a basepoint \(x \in \X \) is fixed. For \(C < \kappa_0\) consider the set \(\mf{S}_C\) consisting of all \(W \in \mf{S}\) such that \(\mathrm{diam} (\fontact W) > C\). Given \(\vec a \in \Omega_\kappa^{\mf{S}_C}\) we define \(\Psi_{\mf{S}_C}(\vec a) = \Psi (\vec b)\), where \(\vec b\) coincides with \(\vec a\) on \(\mf{S}_C\) and 
\(b_U\defeq \pi_U(x)\) for \(U \in \mf{S}-\mf{S}_C\). 
\end{defn}
\begin{remark}
  The choice of basepoint is not very important: the distance formula shows that the Hausdorff distance between 
 the images of \(\Psi_{\mf{S}_C}\) under different choices of basepoints is bounded in terms of \(C\). For this reason, we will suppress the dependence. 
 \end{remark} 
\begin{lemma}\label{lem: the distance formula only sees big domains}
Let \((\mc{X}, \mf{S})\) be a hierarchically hyperbolic space. Then for each \(0\leq C < \kappa\) the spaces \(\Omega_\kappa\) and \(\Omega_\kappa^{\mf{S}_C}\) equipped with the subspace metric are quasi-isometric.
\end{lemma}
\begin{proof}
Setting \(s > C\), the coordinates associated to the elements of \(\mf{S} - \mf{S}_C\) do not contribute to the distance formula. Thus the conclusion follows. \end{proof} 

\Cref{lem: the distance formula only sees big domains} is particularly useful when an HHS satisfies the \emph{bounded domain dichotomy}, that is, when there exists \(C\) such that for each \(U \in \mf{S}\) either \(\mathrm{diam}(\fontact U ) \leq C\) or \(\mathrm{diam} (\fontact U) = \infty\). Notably,  \Cref{rem:uniformbound} states that all HHGs satisfy the bounded domain dichotomy.  The following corollary is immediate.
\begin{coro}\label{coro: bdd domain dich implies q.surjection}
Let \(\mc{X}\) be an HHS satisfying the bounded domain dichotomy, and let \(\bar{\mf{S}}\) consist of all \(W \in \mf{S}\) such that  \(\fontact W\) has infinite diameter. Then there is a constant \(\kappa > 0\) such that \(\Psi_{\bar{\mf{S}}} \colon \Omega_{\kappa}^{\bar{\mf{S}}} \to \mc{X}\) is coarsely surjective, and so \(\Omega_{\kappa}^{\bar{\mf{S}}}\) with the subspace metric is quasi-isometric to \(\mc{X}\). 
\end{coro}

We refer to $U \in \mf{S}$ as being a \emph{(un)bounded domain }when its associated hyperbolic space $\fontact U$ is (un)bounded.  The last ingredient missing to \Cref{prop:gprod} is a criterion to determine when a subspace of an HHS is itself and HHS. 

\begin{defn}\label{defn_HQC}
Let $(\mc{X}, \mf{S})$ be a hierarchically hyperbolic space. A subset $Y\in \mc{X}$ is \emph{hierarchically quasi-convex} if:
\begin{enumerate}
\item the projection $\pi_U(Y)$ are uniformly quasi-convex for all $U \in \mf{S}$;
\item For each $r$ there is an $R$ such that if $x\in \mc{X}$ satisfies $\dist_{_U}(x, \pi_U(Y))\leq r$ for all $U$, then $\dist_{_\cuco{X}}(x,Y)\leq R$.
\end{enumerate}
\end{defn}
By \cite[Proposition~5.6]{BehrstockHagenSisto:HHS2} every hierarchically quasi-convex subset of an HHS can be equipped with an HHS structure. 
We can now prove \Cref{prop:gprod}.  

\begin{proof}[Proof of \Cref{prop:gprod}]
By assumption, \(\bar{\mf{S}}\) can be partitioned as \(\mf{T}_1 \sqcup \cdots \sqcup \mf{T}_n\) where every element of \(\mf{T}_i\) is orthogonal to every element of \(\mf{T}_j\) for \(i \neq j\). By consistency (see Definition \ref{defn: consistent tuple}), the set \(\Omega^{\bar{\mf{S}}}_\kappa\) can be written as \(\Omega^{\mf{T}_1}_\kappa \times \cdots \times \Omega_\kappa^{\mf{T}_n}\). Fix a basepoint \(x \in \mc{X}\) and for each \(\Omega_\kappa^{\mf{T}_i}\) consider the map \(\Psi_{\mf{T}_i} \colon  \Omega_\kappa^{\mf{T}_i} \to \mc{X}\) defined by \(\Psi_{\mf{T}_i}(\vec a) = \Psi_{\bar{\mf{S}}} (\vec b)\), where \(\vec b\) coincides with \(\vec a\) on \(\mf{T}_i\) and is defined to be \(\pi_U(x)\) otherwise. Let \(Y_i\) denote the resulting metric space. The distance formula yields that \(Y_1 \times \cdots \times Y_n\) is quasi-isometric to \(\Psi_{\bar{\mf{S}}}\left(\Omega_\kappa^{\bar{\mf{S}}}\right)\). By \Cref{coro: bdd domain dich implies q.surjection}, the latter coarsely coincides with \(G\).  We are left with proving that each $Y_i$ is hierarchically quasi-convex. By definition, $\pi_W(Y_i)$ coarsely coincides with $\pi_W(\mc{X})$ for $W \in \mf{T}_i$ and it coarsely coincides with $\pi_W(x)$ otherwise. This proves the first item of Definition~\ref{defn_HQC}. For the second, let $y\in \mc{X}$ be such that $\dist_{_W}(y,x)\leq r$ for all $W \not\in\mf{T}_i$. Let $z\in Y_i$ be the realization point of the tuple defined as $\pi_W(y)$ for $W\in\mf{T}_i$ and as $\pi_W(x)$ for $W\not\in \mf{T}_i$. Note that by definition of $\mf{T}_i$ such a tuple is consistent. By the distance formula, we can bound the distance between $z$ and $y$ in terms of $r$, which shows the second item of hierarchical quasi-convexity.
\end{proof}

	\section{Structural results}
	\label{sec:quasilineStructure}

In this section, we give several structural results which will be useful in the proof of \Cref{thm:mainthm}.
 
\begin{lemma} \label{lem:infinitediameter}
	Let $(G,\s)$ be a hierarchically hyperbolic group.  Suppose $\mc U$ is a $G$--invariant collection of pairwise orthogonal domains such that $\fontact U$ has infinite diameter for each $U\in\mc U$. If there exists a domain $V\not \in \mc U$ with $\operatorname{diam}(\fontact V)=\infty$, then for any $U\in\mc U$, we have $U\not\nest V$. 
\end{lemma}

\begin{proof}
	Suppose by way of contradiction that there exists a domain $U\in\mc U$ such that $U\nest V$.  For each $W\in \mc U$, fix any point $p_W\in \fontact W$, and let $p\in G$ be given by partial realization (\Cref{item:dfs_partial_realization} of \Cref{defn:HHS}).   Pick any $g\in G$ and consider the points $\pi_V(g)$ and $\pi_V(p)$.  
	By the choice of $p$, \[\dist_{_V}(p,\rho^U_V)\leq \alpha.\]  Now apply the isometry $\phi_{pg^{-1}}\colon \fontact V\to\fontact (pg^{-1}V)$ induced by $pg^{-1}$.  It follows that \[\dist_{_{pg^{-1}V}}(\phi_{pg^{-1}}(\pi_V(p)),\phi_{pg^{-1}}(\rho^U_V))\leq \alpha.\]  
Since $\phi_{pg^{-1}}(\rho^U_V)$ uniformly coarsely coincides with $\rho^{pg^{-1}U}_{pg^{-1}V}$, we have that $\phi_{pg^{-1}}(\pi_V(p))$ uniformly coarsely coincides with $\rho^{pg^{-1}U}_{pg^{-1}V}$.  

As the action of $G$ on $\s$ fixes $\mc U$ setwise, it follows that $pg^{-1}U\in\mc U$.  Moreover, $pg^{-1}U\nest pg^{-1}V$.  
Thus, by using partial realization as above, we have that $\pi_{pg^{-1}V}(p)$ uniformly coarsely coincides with $\rho^{pg^{-1}U}_{pg^{-1}V}$, and so $\phi_{pg^{-1}}(\pi_V(p))$ uniformly coarsely coincides with $\pi_{pg^{-1}V}(p)$, as well.  Moreover, $\pi_{pg^{-1}V}(p)=\pi_{pg^{-1}V}(pg^{-1}g)$, hence applying the inverse isometry $\phi_{gp^{-1}}$ shows that the distance between $\pi_V(p)$ and $\pi_V(g)$ is uniformly bounded.  Since $g$ was arbitrary and $\pi_V$ is coarsely surjective, it follows that $\fontact V$ has finite diameter, which contradicts our assumption on $V$.
\end{proof}

The following proposition shows that any \(G\)--invariant domain whose associated hyperbolic space is a quasi-line that contains the axis of a loxodromic must be nest minimal.

\begin{prop}
	\label{quasilinesNestMinimal}
	Let $(G,\domains)$ be a hierarchically hyperbolic group, and suppose there exists $U\in \s$ such that $G. U=U$ and $\fontact U$ is $Q$--quasi-isometric to $\R$.  If $G$ contains an element acting by translation on $\fontact U$, then for all $V\propnest U$, $\operatorname{diam}(CV)<\infty$.
\end{prop}

\begin{figure}[h]
	\centering
	\includegraphics[width=\textwidth, trim=90 120 80 140, clip]{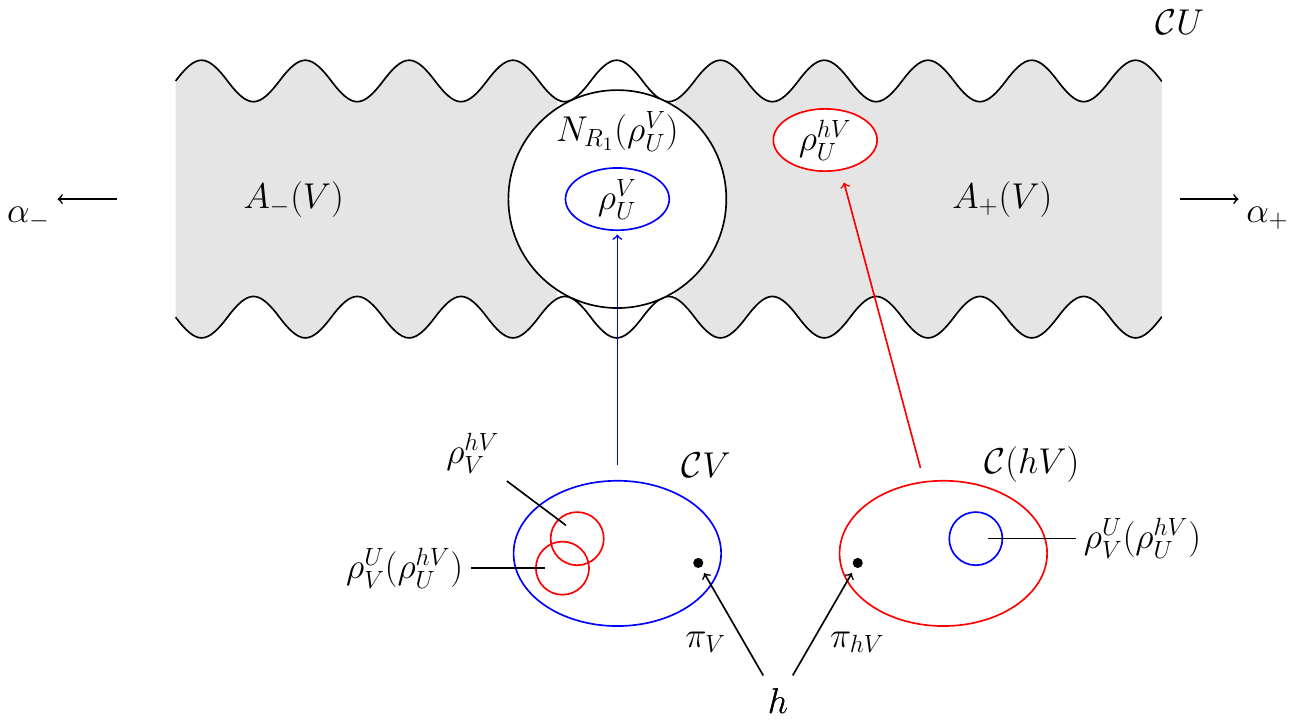}
	\caption{A schematic of the spaces and projections in the proof of \Cref{quasilinesNestMinimal}.}
	\label{fig:Prop32}
\end{figure}

\begin{proof}
	We remark that since we are solely concerned with understanding the spaces \(\fontact W\) for \( W \in \mf{S}\), we can fix an arbitrary generating set to work with for the proof of this proposition.  This assumption is only needed to prove \Cref{eq:Claim2} below.
	 
	Let \(D\) be the hierarchy constant introduced in \Cref{rem:hierarchyConstant}, and let $\kappa_1$ be the constant from \cite[Proposition~1.8]{BehrstockHagenSisto:HHS2}.  Let $\partial\fontact U = \braces{\alpha_{+}, \alpha_{-}}$.  For any domain $V$ which properly nests into $U$, the nesting axiom (\Cref{item:dfs_nesting}) gives that $\diam_{\fontact U}(\rho^V_U) \leq D$.  The hyperbolic space $\fontact U$ is a $Q$--quasi-line for some constant $Q$; we  may assume without loss of generality that $Q > 1/\sqrt{2}$. It follows that there is a constant $R_1 > 2D+\kappa_1$ such that the neighborhood $N_{R_1}(\rho^V_U)$ disconnects $\fontact U$.  Let $A_+(V)$ and $A_-(V)$ be the two connected components of $\fontact U \smallsetminus N_{R_1}(\rho^V_U)$ containing $\alpha_{+}$ and $\alpha_{-}$, respectively, and let $A_\pm(V) = A_{+}(V) \bigcup A_{-}(V)$ denote their union.  See Figure~\ref{fig:Prop32}.  Since $\fontact U$ is a path connected $Q$--quasi-line by assumption, we have  $\diam_{\fontact U}\pns{\fontact U \smallsetminus (A_\pm(V))} \leq 2\pns{Q^2R_1 + Q^2 + Q}$.   
	Let $R_2=2(Q^2R_1 + Q^2 + Q)$, and note that $R_2>D+\kappa_1$.  
	The bounded geodesic image axiom (\Cref{item:dfs:bounded_geodesic_image}) states that every geodesic segment in $A_{+}$ or $A_{-}$ projects to $\fontact V$ with diameter at most $D$.  Thus $\diam_{\fontact V}(\rho^U_V(A_{+}(V))) \leq 2D$ and $\diam_{\fontact V}(\rho^U_V(A_{-}(V))) \leq 2D$.  
	 Since the map $\rho^V_U$ is $G$--equivariant, we have $A_\pm(V')=hA_\pm(V)$ whenever $V'=hV$. 
	
	The proof follows by contradiction using the following two claims, each relying on the assumption that there is a domain properly nested into $U$ whose curve graph has infinite diameter.  

	\begin{description}
		\item[Claim 1] 
			If $V' \propnest U$ and $\fontact V'$ is unbounded, then for all $L > 0$ there is an unbounded domain $V \propnest U$ such that
			\begin{equation}\label[claim]{eq:Claim1}
				\dist_{_V}\pns{1, \rho^U_V(A_\pm(V))} > L.
			\end{equation}
		\item[Claim 2] 	
			If $V \propnest U$ and $\fontact V$ is unbounded, then for all $L > 0$ there is an element $h \in G$ such that 
			\begin{equation}\label[claim]{eq:Claim2}
				\dist_{_V}\pns{h, \rho^U_V(A_\pm(V))} > L \qquad \tand \qquad \dist_{_U}\pns{\rho^{hV}_U, \rho^V_U} > L.
			\end{equation}
	\end{description}

	We complete the proof assuming the claims, which will be addressed later.  Take $L= R_2$, and suppose there is a domain $V$ that properly nests into $ U$ such that $\fontact V$ has unbounded diameter.  By \Cref{eq:Claim1}, we may assume without loss of generality that $V$ satisfies \eqref{eq:Claim1}.  
The second statement of (\ref{eq:Claim2}) and	 \Cref{lem:closerhoimpliestrans}   give that $V \pitchfork hV$.  Since $G. U = U$, every element $g \in G$ acts on $\fontact U$ by isometries.  
We have $\diam_{\fontact U}(\fontact U \smallsetminus A_\pm(V)) = \diam_{\fontact U}(\fontact U \smallsetminus hA_\pm(V))  \leq R_2$,  and consequently the second statement of \Cref{eq:Claim2} implies that $\rho^{hV}_U \subset A_\pm(V)$ and $\rho^V_U \subset hA_\pm (V)$.  
	
	 The second statement of \Cref{eq:Claim2} and \cite[Proposition~1.8]{BehrstockHagenSisto:HHS2} show that $\diam_{\fontact V}(\rho^{hV}_V\cup \rho^U_V(\rho^{hV}_U))\leq \kappa_1$.  As $\rho^U_V(\rho^{hV}_U)\subseteq \rho^U_V(A_\pm(V))$,
	we thus have
\begin{equation}\label{eqn:Vdist}
	    \dist_{_V}\pns{h, \rho^{hV}_V} \geq \dist_{_V}\pns{ h, \rho^U_V(A_\pm(V)) } - \kappa_1
	    > R_2- \kappa_1 \geq D,
	    \end{equation}
 where the second to last inequality follows from the first statement of \Cref{eq:Claim2}.
Applying the fact that the projections $\rho$ are $G$--equivariant (Remark \ref{rem:rhocommutes}) to  \eqref{eq:Claim1} yields 
	$$ \dist_{_{hV}}\pns{h,\rho^U_{hV}(hA_\pm (V)) } > L = R_2. $$
	    Since $\rho_U^V \subseteq hA_{\pm}(V)$ and $\dist_{_V}(1, \rho_V^{U}(A_{\pm}(V)) =  \dist_{_{hV}}(h, \rho_{hV}^{U}(hA_{\pm}(V)) > R_2$, an analogous argument yields
\begin{equation}\label{eqn:hVdist}
	    \dist_{_{hV}}\pns{h, \rho^V_{hV}} \geq \dist_{_{hV}}\pns{ h, \rho^U_{hV}(hA_\pm(V)) } - \kappa_1 > R_2 - \kappa_1 \geq D.
\end{equation}

	However, the  inequalities \eqref{eqn:Vdist} and \eqref{eqn:hVdist} contradict the transversality and consistency axiom (\Cref{item:dfs_transversal}) applied to $h$ projected to $V$ and $hV$, which states that 
	$$ \min\braces{\dist_{_V}\pns{h, \rho^{hV}_V}, \dist_{_hV}\pns{h, \rho^V_{hV}} } \leq D. $$
	It remains to prove the two claims.  
	
	\begin{description}
		\item[Proof of \Cref{eq:Claim1}] 
			Let $L > 0$ be fixed, and consider \(A_{\pm}(V')\). We have that $\rho^U_{V'}(A_\pm(V'))$ has bounded diameter. 
			If $\dist_{_{V'}}\pns{ 1, \rho^U_{V'}(A_\pm(V'))} > L$ then we are done by taking $V = V'$.  Otherwise, $\dist_{_{V'}} \pns{1, \rho^U_{V'}(A_\pm(V')) } \leq L$.
			Since $\rho^U_{V'}(A_\pm(V'))$ is bounded and $\pi_{V'}$ is $D$--coarsely surjective there is an element $g^{-1} \in G$ so that 
			$ \dist_{_{V'}}\pns{ g^{-1}, \rho^U_{V'}(A_\pm(V'))} > L + D $.
			By equivariance, we can apply $g$ to get
			$ \dist_{_{gV'}}\pns{ 1, \rho^U_{gV'}(gA_\pm(V')) }>L+D. $
			Taking $V = gV'$ completes the claim.  
			
		\item[Proof of \Cref{eq:Claim2}]
			
			Let $L > D$ be fixed exceeding the hierarchy constant and $t \in G$ be an element acting by translation on $\fontact U$, which exists by assumption. Let \(\gamma\) be any isometry of $\fontact U$ that fixes the endpoints and 
			moves some point $x_0 \in \fontact U$ less than $L$.   Then there is a constant $\bar{L} \geq L$  depending only on the quasi-line constants of $\fontact U$ (and not on the choice of $\gamma$) such that $\gamma$ moves every point of $\fontact U$ by at most $\bar{L}$.
			\\
			Let \(\hat{G} \leq G\) be the index 2 subgroup of \(G\) that fixes \(\partial \fontact U\) pointwise. Note that $t$ acts as translation, and so  $t \in \hat{G}$.  Moreover, since \(G\) coarsely surjects onto \(\fontact U\), so does \(\hat{G}\). 
			Pick $M > 0$ so that $M\tau_0 > 2\bar{L} + D$, where $\tau_0$ is as in \Cref{lem:translengthbound}.  As before, coarse surjectivity guarantees the existence of an element $h' \in \hat G$ satisifying $$\dist_{_U} (h', \rho^U_V(A_\pm(V))) > \bar{L} + KM\vert t \vert + K, $$ where $\pi_V$ is $K$--coarsely Lipschitz and \(\vert t \vert\) is the word length of \(t\) in the fixed generating set.  If $\dist_{_U} (\rho^{h'V}_U, \rho^V_U) > L $ then we are done by taking $h = h'$, so assume $\dist_{_U} (\rho^{h'V}_U, \rho^V_U) \leq L$.
			
			Consider $h = h't^M$.  Using the fact  that  \(\pi_V\) is Lipschitz and the triangle inequality, we have
			\begin{align*}
				\dist_{_V}\pns{h, \rho^U_V(A_\pm(V)) }
					&\geq \dist_{_V}\pns{h', \rho^U_V(A_\pm(V)) } - \dist_{_V}\pns{h't^M, h'}
					\\
					&\geq (\bar L + KM\vert t \vert + K) - (KM\vert t \vert + K)
					\\
					&\geq \bar L 
					 \geq L.
			\intertext{Thus the first statement of the claim holds.  By the choice of $\bar{L}$, we have that \(\dist_{_U} (x, h'x) \leq \bar{L}\) for all \(x \in \fontact U\). Thus}
				\dist_{_U}\pns{ \rho^V_U, \rho^{hV}_U }
					&\geq \dist_{_U}\pns{\rho^{V}_U, h't^M \rho^{V}_U } - D \\
					& \geq 
					\dist_{_U}\pns{\rho^V_U, t^M\rho^{V}_U } - \dist_{_U} \pns{t^M\rho^{V}_U, h't^M\rho^{V}_U }- D
					\\
					&\geq (2\bar{L} + D) - \bar{L} - D
					\\
					& \geq \bar L \geq L,
			\end{align*} completing the proof of \Cref{eq:Claim2}.  \qedhere
	\end{description}
\end{proof}

Next, we give a sufficient condition for when a collection of pairwise orthogonal domains have associated hyperbolic spaces that are quasi-lines.

\begin{prop} \label{prop:quasilines}
Let \((G,\s)\) be a hierarchically hyperbolic group and $U\in \s$ a domain such that there is a pair of points $\alpha,\beta\in \partial \mc CU$ which are fixed pointwise by $G$.  Then $\fontact U$ is quasi-isometric to a line.
\end{prop}
\begin{proof}
	Let \(\gamma\) be a geodesic between the points \(\alpha, \beta \in \partial \fontact U\), and let \(h\in G\). We want to uniformly bound \(\dist_{_U}(h, \gamma)\). Since there exists \(C= C(\mf{S})\) such that \(\pi_U\) is \(C\)--coarsely surjective, this would prove the result. 
	Let \(g\in G\) be such that \(\dist_{_U}(g, \gamma)\leq C\), and consider \(hg^{-1} \gamma\). Since all the generators fix \(\alpha, \beta \in \partial \fontact U\), we have that \(hg^{-1}\gamma\) is a geodesic of \(\fontact U\) with the same endpoints as \(\gamma\). By the hyperbolicity of $\fontact U$, the Hausdorff distance between $\gamma$ and $hg^{-1}\gamma$ is uniformly bounded. Moreover, by equivariance of the map \(\pi_U\) we have \(\dist_{_U}(h, hg^{-1}\gamma)=\dist_{_U}(g,\gamma) \leq C\), which implies that $\dist_{_U}(h,\gamma)$ is uniformly bounded, concluding the proof.
\end{proof}

We end this section by describing domains which are transverse to a $G$--invariant domain whose associated hyperbolic space has infinite diameter.  

\begin{prop}\label{prop:G_inv_implies_corasely_surj}
Let \((G, \mf{S})\) be a hierarchically hyperbolic group and suppose  there is a \(G\)--invariant domain \(U \in \mf{S}\) such that \(\mathrm{diam} (\fontact U ) = \infty\).  For any $W\in\s$ satisfying \(W \pitchfork U\), the space \(\fontact W\) has uniformly bounded diameter. 
\end{prop} 

\begin{proof}
Let $\Omega_\kappa\subset \Pi_{W\in\s}2^{\fontact W}$ and $\Phi\colon \Omega_\kappa\to 2^\X$ be as in \Cref{sec:coordsys}.
Let \(\kappa \geq \kappa_1\) and let \(Y\) be the subset of \(\Omega_\kappa\) consisting of all tuples whose \(W\)--coordinate is \(\rho_W^U\) for each \(W \pitchfork U\). Since \(\fontact U\) has infinite diameter, \(\Phi(Y)\) is an infinite diameter subset of \(G\). Moreover, since \(U\) is \(G\)--invariant, so are \(Y\) and \(\Phi(Y)\). Since \(G\) acts coboundedly on itself, we have that \(\Phi (Y)\) coarsely coincides with \(G\). Since \(\Phi\) is a quasi-isometry, we conclude that \(Y\) coarsely coincides with \(\Omega_\kappa\). Thus, the spaces \(\fontact W\) are uniformly bounded for every \(W \pitchfork U\).
\end{proof}

\section{Proof of Main Theorem}
\label{sec:proof}

\Cref{thm:mainthm} follows immediately from \Cref{shortFreeImpliesUEG} and the following theorem.  Recall that for any generating set $\genset$ and any \(n \geq 1\), we denote by $\genset^n$ the ball of radius $n$ about the identity in the Cayley graph of $G$ with respect to $\gener$.

\begin{theorem}
	\label{thm:technicalmainthm}
	Let $(G,\s)$ be a virtually torsion-free hierarchically hyperbolic group.  Then there exists a constant $M > 0$ depending only on $(G,\s)$  such that one of the following occurs.
	\begin{enumerate}[(a)]
	\item 
		$G$ is virtually abelian.
		
	\item 
		For any generating set $\genset$, there are elements $u,w \in \Ball{M}$ which form a basis for a free sub-semigroup.
		
	\item 
		There is a $G$--invariant collection $\overline{\mc B}$ of pairwise orthogonal domains such that  $G$ is quasi-isometric to $\Z^{|\overline{\mc B}|} \times E$, where $E$ is a non-elementary space.
		Moreover, $G$ has a generating set all of whose members act elliptically on $E$.
	\end{enumerate}
\end{theorem}

For the remainder of the paper, we fix a finitely generated torsion-free hierarchically hyperbolic group $G$ and a generating set \(\gener\) for \(G\), with the convention that \(\gener\) contains the identity. 

Recall that \(N\) is the maximal number of pairwise orthogonal domains of \(G\).
Let
\[  
	\mc B = \bigcup\limits_{s\in \gener}\B(s)
\]
be the collection of domains onto whose associated hyperbolic spaces the axes of the generators have unbounded projection 
and let
\begin{equation}\label{eqn:Bbar}
\bar{\mc B} = \gener^N . \mc B
\end{equation}
be the set of images of these domains under words of length at most $N$.  Note that since $\genset$ is finite and $|\B(s)|\leq N$ for all $s\in\genset$, it is always the case that $\overline{\mc B}$ is a finite set.  Moreover, since $G$ is torsion-free, $\bigset{s}$ is non-empty for every $1 \neq s\in X$ (see \Cref{rmk:nonemptybigset}), and therefore $\mc B\neq\emptyset$.

After first passing to a torsion-free finite index subgroup, the proof of \Cref{thm:technicalmainthm} will be divided into two main cases using the following proposition. 

\begin{prop}\label{prop: basic combinatorics}
  	Let \(G\) be a torsion-free hierarchically hyperbolic group, \(\gener\)  a generating set for \(G\) containing the identity, and
  	\(N \)  the maximal number of pairwise orthogonal domains. Then one of the following holds. 
   	\begin{enumerate}
		\item 
		\label{item: combinatorial cases, find non orthogonal}
		There are elements \(s,t \in \gener^{2N+1}\) such that \(\bigset{s}\) and \(\bigset{t}\) contain two non-orthogonal elements;
		\item 
		\label{item: combinatorial cases, fixed set}
		The set $\bar{\mc B}$ defined in (\ref{eqn:Bbar}) is a finite collection of pairwise orthogonal domains stabilized by $G$ in the action on $\domains$.
		
   		\end{enumerate}
		Moreover, if Item \ref{item: combinatorial cases, fixed set} holds, then there is a finite index subgroup, $\hat{G} \leq G$, of index at most $N!$ fixing $\bar{\mc B}$ pointwise.
\end{prop}

\begin{proof}
Suppose that Item~\ref{item: combinatorial cases, fixed set} does not hold. Then either $\overline{\mc B}$ contains two non-orthogonal elements or $\overline{\mc B}$ is not a $G$--invariant set.
Suppose that $\overline{\mc B}$ is not $G$--invariant.  Then $X$ does not fix $\overline{\mc B}=X^N.\mc B$ setwise, and thus $X$ does not  fix $X^k.\mc B$ setwise for any $1\leq k\leq N$.  Hence
for each \(1\leq k \leq N\), 
\[
\gener^{k}. \mc{B} \neq \gener^{k+1}. \mc{B}.
\] 
Since the identity is contained in \(\gener\), we have
\[ \gener^{k}. \mc{B} \subseteq \gener^{k+1}. \mc{B}.\]
In particular, since $\mc B\neq\emptyset$, this implies that \(|\gener^{N}. \mc{B} | \geq  N +1\). However, this is a contradiction, as there can be at most \(N\) pairwise orthogonal elements. We conclude that if Item \ref{item: combinatorial cases, fixed set} does not hold, then there must be non-orthogonal domains \(V_1, V_2 \in \overline{\mc B}\). Thus for $i=1,2$ there are generators \(s_i \in \gener\), domains \(U_i \in \bigset{s_i}\) and elements \(g_i \in \gener^{N}\) such that 
\[V_i = g_i U_i.\]
This implies that \(V_i \in \bigset{g_i s_i g_i^{-1}}\). If we denote the word length with respect to the generating set $X$ by \(| \cdot |_\gener\), we have \[|g_i s_i g_i^{-1}|_\gener \leq |g_i|_\gener + |s_i|_\gener + |g_i^{-1}|_\gener \leq 2N +1,\] and  Item~\ref{item: combinatorial cases, find non orthogonal} follows by setting $s=g_1s_1g_1^{-1}$ and $t=g_2s_2g_2^{-1}$. 

Finally, suppose  Item~\ref{item: combinatorial cases, fixed set} holds, that is, suppose that $\bar{\mc B}$ is a finite collection of pairwise orthogonal domains stabilized by $G$ in the action on $\domains$.  By definition of \(N\), we have $|\bar{\mc B}| \leq N$.  This induces a map to the symmetric group $G \to \Sym(N)$ whose kernel is a subgroup of $G$ of index at most $N!$ fixing $\bar{\mc B}$ pointwise, which establishes the final statement of the proposition.
\end{proof}

We address the two cases of \Cref{prop: basic combinatorics} in separate subsections. 
		\subsection{Case 1}\label{sec:case1}
Assume that  Item \ref{item: combinatorial cases, find non orthogonal} of \Cref{prop: basic combinatorics} holds, that is, there exist elements $s,t\in\Ball{2N+1}$ and domains $U\in\B(s)$ and $V\in\B(t)$ such that $U\not\perp V$. 
There are two possibilities in this case: either $U\trans V$ or $U\propnest V$ (the case $V\propnest U$ is completely analogous).  We deal with each possibility in a separate proposition and will demonstrate that in each case there are uniform powers of $s$ and $t$ which generate a free subgroup.

\begin{prop} 
	\label{case:TransverseDomains}
		Let $s$ and $t$ be a pair of axial HHS automorphisms with domains $U\in \B(s)$ and $V\in\B(t)$ such that $U\trans V$.  There exists a constant $k_1$ depending only on $(G,\s)$ such that $
	\abrackets{ s^{k_1},t^{k_1} } \isom F_2$.
\end{prop}
\begin{proof}
	By passing to a uniform power $(2N+1)!$, 
	we may assume that $\B(s)$ and $\B(t)$ are fixed pointwise by $s$ and $t$, respectively. 
	
	Let $\kappa_0$ be the constant from \Cref{item:dfs_transversal} (Transversality) of \Cref{defn:HHS}.  We will apply the ping-pong lemma to the following subsets of $G$:
	\[
		Y_s = \braces{ x\in G : \dist_{U}\pns{\pi_U(x), \rho^V_U} > \kappa_0 }
	\qquad
	\tand  
	\qquad
		Y_t = \braces{ x\in G :  \dist_{V}\pns{\pi_V(x), \rho^U_V} > \kappa_0 }.
	\]

	Transversality and consistency imply that these sets are disjoint.  Note that for all $W,T \in \s$, the projection map $\pi_W: G \to \fontact W$ is coarsely surjective and $\rho^T_W$ is a bounded subset of $\fontact W$ whenever $T\trans W$. Since $\fontact U$ and $\fontact V$ are infinite diameter, this implies that $Y_s$ and $Y_t$ are non-empty.
	
	Let $\displacementBound$ be the minimal translation length from \Cref{lem:translengthbound}.  Fix a constant $k \geq 2\kappa_0\displacementBound^{-1}$ and a point $x \in Y_s$.  By  transversality and consistency, we have $\dist_{_V}(x,\rho^U_V)\leq \kappa_0$. Using this fact in addition to \Cref{lem:translengthbound} and the triangle inequality, we have
	\begin{align*}
    	\dist_{_V}\pns{\rho^U_V, t^{k(2N+1)!}. x} 
    	&\geq \dist_{_V}\pns{x, t^{k(2N+1)!}. x} - \dist_{_V}\pns{x, \rho^U_V} \\
    	&\geq \displacementBound\abs{k} - \dist_{_V}\pns{x, \rho^U_V}\\
    	&\geq 2\kappa_0 - \kappa_0  \\
    	&= \kappa_0
	\end{align*}

Thus $t^{k(2N+1)!}. x\in Y_t$, and so $t^{k(2N+1)!}(Y_s)\subseteq Y_t$. Observe that the only requirement for \(k\) was \(k \geq 2\kappa_0\displacementBound^{-1}\). In particular, the conclusion holds for multiples of \(k\). The same argument works with negative powers, so $t^{-k(2N+1)!}(Y_s)\subseteq Y_t$.  By a symmetric argument, it follows that $s^{\pm k(2N+1)!}(Y_t)\subseteq Y_s$.  Thus, by the ping-pong lemma $\abrackets{s^{k(2N+1)!}, t^{k(2N+1)!}} \isom F_2$. Setting $k_1=2\kappa_0\tau_0^{-1}(2N+1)!$ completes the proof.
\end{proof}

We note that in the previous proposition (and in many of the later results), if we allow $s$ and $t$ to have different exponents, then we can find smaller constants $k_{1,s}$ and $k_{1,t}$ such that $\abrackets{ s^{k_{1,s}},t^{k_{1,t}} } \isom \F_2$.  In particular, we may take $k_{1,s}=2\kappa_0\displacementBound^{-1}m_s$ and $k_{1,t}=2\kappa_0\displacementBound^{-1}m_t$, for some $m_s,m_t\leq N$. Also, the stabilization power $(2N+1)!$ is not optimal since it is given by the kernel of a map from a copy of $\Z$ to a cyclic subgroup of $\Sym(2N+1)$, which can have size at most $LCD(1,2, \dotsc, 2N+1) $, which grows slower than factorial.
For ease of notation, however, we choose to use the larger uniform exponent.

We now turn to the second possibility in Case 1.

\begin{prop} 
	\label{case:ProperNesting}
	Let $s$ and $t$ be a pair of axial HHS automorphisms with domains $U\in \B(s)$ and $V\in \B(t)$ such that $U$ is properly nested in $V$.
	There exist constants $k_2$ and $n_0$ depending only on $(G,\s)$ such that $\abrackets{ s^{k_2}, t^{n_0} s^{k_2} t^{-n_0} } \isom \F_2$. 
\end{prop}

\begin{proof}
	Since $U$ is properly nested into $V$, the projection $\rho^U_V$ in $\fontact V$ satisfies $\diam_{\fontact V}\pns{\rho^U_V} \leq D$.  
	Recall that  \(\dist_{_V}(t^i. \rho^U_V,\rho^{t^i. U}_V)\leq \kappa_0\) for all $i$.  By \Cref{lem:translengthbound}, there is a uniform power $n_0$ of $t$ such that \(\dist_{_V}(\rho^{t^{n_0} . U}_V, \rho_V^U)\geq 10D\).  In particular, we can take any $n_0 \geq 10D\displacementBound^{-1}$.  By \Cref{lem:closerhoimpliestrans}, this implies that  $(t^{n_0}. U) \trans U$.  Applying \Cref{case:TransverseDomains} to the pair \(s, t^{n_0}st^{-n_0}\) and replacing $2N+1$ with $2n_0+1$ yields the desired constant $k_2$, which completes the proof. 
\end{proof}

		\subsection{Case 2}\label{sec:case2}
		Recall that   
 \[\mc B=\bigcup_{s\in \genset} \B(s),\qquad \overline{\mc B}=\genset^N. \mc B,\] and \[\hat{G}= \ker \pns{ G \to \Sym(N) }.\]
We now suppose that Item \ref{item: combinatorial cases, fixed set} of \Cref{prop: basic combinatorics} holds, that is, $\overline{\mc B}$ is a finite collection of pairwise orthogonal domains which is stabilized by the action of $G$ on $\s$ and fixed pointwise by the action of $\hat G$ on $\s$.

\begin{prop} \label{prop:SinB}
Suppose $\fontact S$ has infinite diameter.  Then either there exists a constant $k_3$ depending only on $(G,\s)$ and elements $s,t\in X$ 
such that $\langle s^{k_3},ts^{k_3}t^{-1}\rangle\cong F_2$ or $G$ is virtually cyclic.
\end{prop}

\begin{proof}
	
Let $U\in \overline{\mc B}$.  Then, by definition, there exists $h\in G$ with $|h|_X\leq N$, a generator  $x\in\genset$, and a domain $W\in\B(x)$ such that $U=h.W$.  As $\fontact W$ has infinite diameter and $h$ acts as an isometry on the associated hyperbolic spaces, $\fontact U$ must have infinite diameter, as well.

 Since $\fontact S$ has infinite diameter by assumption and $U\nest S$, it follows from \Cref{lem:infinitediameter} applied with $\mc U=\overline{\mc B}$ that $S\in\overline{\mc B}$.    By definition, $S \in \overline{\mc B}$ implies that $S = g.V$ for some $g\in G$ with $\abs{g}_X \leq N$ and $V \in \bigset{s}$ for some $s \in X$.  However, hierarchical automorphisms preserve the $\nest$--levels of elements of $\s$, and $S$ is the unique $\nest$--maximal domain in $\s$. Thus, $S=g.V$ if and only if $V$ has the same level as $S$, and we conclude that $V = S$.  This implies that $S\in\bigset{s}$. (In fact, this implies that $S=\bigset{s}$ by \cite[Lemma~6.7]{DurhamHagenSisto:Boundaries}, but we will not need this stronger statement.)
 
	The action of $G$ on $\fontact S$ is cobounded and acylindrical by \cite[Corollary~14.4]{BehrstockHagenSisto:HHS1}.  Let $E(s)$ denote the stabilizer of the endpoints of the axis of $s$ in $\partial \fontact S$.  If for every generator $r\in\genset$ we have $r\in E(s)$, then $G$ is virtually cyclic by \cite[Lemma~6.5]{DahmaniGuirardelOsin}. 
	
	Otherwise, there exists a generator $t \in X\setminus\{s\}$ such that $t\not\in E(s)$, and hence $t$ does not stabilize the endpoints of the axis of $s$ in $\partial\fontact S$. In particular, $|\partial\fontact S|\geq 3$, that is,  $\fontact S$ is a non-elementary hyperbolic space.  

By \cite[Corollary~6.6]{DahmaniGuirardelOsin}, $t\not\in E(s)$ if and only if $ts^nt^{-1} \neq s^{\pm n}$ for any $n \neq 0$.  Therefore, with the above choice of $s$ and $t$, \Cref{Fujiwara:freeSubgroup} guarantees the existence of a constant $k_3$ such that $\abrackets{s^{k_3}, ts^{k_3}t^{-1}} \cong F_2$.
\end{proof}

In particular, the proof of \Cref{prop:SinB} shows that whenever $\fontact S$ has infinite diameter there exist two uniformly short elements which are independent loxodromic elements with respect to the action on $\fontact S$.

We are now ready to prove \Cref{thm:technicalmainthm}.

\begin{proof}[Proof of \Cref{thm:technicalmainthm}]

Consider the finite-index torsion-free subgroup $H$ of $G$.  Then $(H,\s)$ is a normalized HHG by \Cref{lem: passing to f.i.subgroup}, and by \Cref{finiteIndexWordLength}, there is a generating set $X'$ for $H$ all of whose elements have $X$--length at most $2d - 1$, where $d=[G:H]$. This means that if we can prove the desired trichotomy for \(H\), it will follow for \(G\).  Thus, we can and will assume that $G$ is torsion-free.

Let $k_1$ be the constant from \Cref{case:TransverseDomains}, 
$k_2$ and $n_0$  the constants from \Cref{case:ProperNesting},  
$k_3$  the constant from \Cref{prop:SinB}, $\delta$  the hyperbolicity constant of $\fontact U$ for any $U\in \s$, and $\tau_0$  the constant from  \Cref{lem:translengthbound}.  Also let 
\[
	k_4 = \ceil{10000\delta\tau_0^{-1}},
\]
and
\[
	M \geq \max\{k_1, 2n_0+k_2, k_3+2, 3(k_4+2)(N+1)! \}.
\]
 We recall that our goal is to show that one of the following occurs:  \begin{enumerate}[(a)]
 \item \label{itema} \(G\) is virtually abelian; 
\item \label{itemb} there exist two words of length at most \(M\) that generate a free semigroup; or 
\item \label{itemc} \(G\) is quasi-isometric to a product \(\mathbb{Z^{|\overline{\mathcal B}|}}\times E\), where \(E\) has infinite diameter and is not quasi-isometric to \(\mathbb{Z}^n\). 
\end{enumerate}
Let $X$ be an arbitrary generating set for $G$. Then one of the two cases of \Cref{prop: basic combinatorics} must occur. If the hypotheses of Case 1 are satisfied, then (\ref{itemb}) holds by either \Cref{case:TransverseDomains} or \Cref{case:ProperNesting}. So, suppose Case 2 occurs and the set 
\[
	\bar{\mc B} = \genset^N.\mc{B}
\]
defined in (\ref{eqn:Bbar}) is  fixed setwise by \(G\). 

If \(S \in \bar{\mc{B}}\), then $\fontact S$ has infinite diameter, and so (\ref{itema}) or (\ref{itemb}) holds by \Cref{prop:SinB}. If \(S\not\in \bar{\mc{B}}\), then \(\mathrm{diam} (\fontact S) < \infty\) (in particular, it is uniformly bounded) by applying \Cref{lem:infinitediameter} with $\mc U=\overline{\mc B}$. 

By passing to a further finite index subgroup, we can assume
that \(\bar{\mc{B}}\) is fixed pointwise by \(G\). Indeed, consider the subgroup $\hat{G} = \ker(G \to \Sym(\bar{\mc B}) ) $ of index at most $N!$ which fixes \(\bar{\mc{B}}\) pointwise. By \Cref{finiteIndexWordLength}, there is a generating set $Y'$ for $\hat{G}$ all of whose elements have $X$--length at most $2N! - 1$. This means that if we can prove the desired trichotomy for \(\hat{G}\), it will follow for \(G\).
By definition, every domain $U \in \bar{\mc B}$ supports the axis of at least one element in $X^{2N + 1}$.  Observe also that, by \Cref{prop:HHGisomClassification} there is a constant $K$ between 0 and $N!$ such that $g^K \in \hat{G}$.  Expand the generating set for $\hat{G}$ to be
\[
	Y = Y' \bigcup \braces{g^K : g \in X^{2N+1}}.
\]
Elements of $Y$ have $X$--length at most $(2N+1)N! < 3(N+1)!$.
Since each domain of $\bar{\mc{B}}$ was in the big set of some element of $X^{2N+1}$, each domain is also in the big set of some element of $Y$.

For the rest of the proof, we restrict our attention to $\hat G$, which acts on \(\fontact U\) for each \(U \in \bar{\mc{B}}\). For each \(U \in \bar{\mc{B}}\), there exists an element \(s_U\in G\) with $|s_U|_X\leq 2N+1$ that acts loxodromically on \(\fontact U\). Thus $s_U^K\in \hat G$ also acts loxodromically on $\fontact U$, and  $|s_U^K|_Y=1$, by the definition of $Y$.  Let \(s_U^\pm\) be the fixed point of \(s_U^K\) on \(\partial \fontact U\). 
We claim that either (\ref{itemb}) holds or all the generators fix \(\{s_U^+, s_U^-\}\) setwise. Indeed, if \(t\) is an element of $Y$ that does not fix \(\{s_U^+, s_U^-\}\), the conjugate $t^{-1}s_U^Kt$ is an independent loxodromic with respect to the action on $\fontact U$. 
By \Cref{lem:translengthbound} there is a uniform lower bound on the translation length of $s_U^K$ (which is equal to the translation length of $t^{-1}s_U^Kt$) with respect to the action on $\fontact U$.   
Therefore, \Cref{BreuillardFujiwara:freeSemigroup} implies that for $k_4$ defined as above, some pair in $\{(s_U^K)^{\pm k_4},t^{-1}(s_U^K)^{\pm k_4}t\}$ generates a free semigroup, and hence (\ref{itemb}) holds. 

Thus, we may assume that for each \(U \in \bar{\mc{B}}\), the set \(\{s_U^+, s_U^-\}\) is \(\hat G\)--invariant. By  \Cref{prop:quasilines}, we conclude that \(\fontact U\) is a quasi-line for each \(U \in \bar{\mc{B}}\). Let $\overline{\mc B}=\{U_1,\dots, U_n\}$ for some $n$, and let \(\bar{\mf{S}} = \{V \in \mf{S} \mid \mathrm{diam}( \fontact V) = \infty\}\). We claim that \(W \bot U_i\) for each \(W \in \overline{\mf{S}} - \bar{\mc{B}}\) and for all \(i\). To see this, suppose that \(\fontact W\) is unbounded. Then \Cref{lem:infinitediameter} and \Cref{quasilinesNestMinimal} imply that for each $i$, either \(W \bot U_i\) or \(W \pitchfork U_i\). Since \(U_i\) is \(\hat G\)--invariant, by \Cref{prop:G_inv_implies_corasely_surj}, we must have \(W \bot U_i\). 
Thus, we can partition \(\bar{\mf{S}}\) into pairwise orthogonal sets as follows:
\[\bar{\mf{S}} = \{U_1\} \sqcup \cdots \sqcup \{U_n\} \sqcup \left(\bar{\mf{S}} - \bar{\mc{B}} \right).\]

Let $\Omega_{\kappa}^{\s}$ be as in \Cref{sec:coordsys}.
By \Cref{prop:gprod}, we conclude that \(\hat G\) (and therefore $G$) is quasi-isometric to \(\mathbb{Z}^{|\bar{\mc{B}}|} \times \Omega_{\kappa}^{\bar{\mf{S}}- \bar{\mc{B}}}.\)  If $\Omega_{\kappa}^{\bar{\mf{S}}- \bar{\mc{B}}}$ is quasi-isometric to $\Z^m$ for some $m$, then (\ref{itema}) holds.  
Otherwise, (\ref{itemc}) holds with respect to the initial generating set, $\genset$, and $E = \Omega_{\kappa}^{\bar{\mf{S}}- \bar{\mc{B}}}$, completing the proof.  
\end{proof}

\Cref{thm: the non ueg case} follows immediately from the proof of \Cref{thm:technicalmainthm}.

\subsection{Applications} 
\label{subsec:applications}
We begin by proving Corollaries \ref{cor:ac}, \ref{cor:qcx}, and \ref{cor:CSnonelem} from the introduction, whose statements we recall for the convenience of the reader.

\coroAsyCone

\begin{proof}[Proof of \Cref{cor:ac}]
Let $G$ be a non-virtually cyclic virtually torsion-free hierarchically hyperbolic group. 
It follows from \cite[Proposition 1.1]{DMS}  that having a cut-point in an asymptotic cone of $G$ is equivalent to $G$ having super-linear divergence.  However, this cannot occur if $G$ is quasi-isometric to a product with unbounded factors, and therefore $G$ has uniform exponential growth by \Cref{thm:mainthm}.  

The second statement follows from 
\cite[Proposition~3.24: (1)$\iff$(2)]{DMS}, which show that if a geodesic metric space $X$ has an unbounded Morse quasi-geodesic, then every asymptotic cone of $X$ has a cut-point.  
\end{proof}

\coroQuasiConvex

\begin{proof}[Proof of \Cref{cor:qcx}]
Let $G$ be a non-virtually cyclic virtually torsion-free hierarchically hyperbolic group, and let $H\leq G$ be an infinite quasi-convex subgroup of infinite index.  If $G$ is quasi-isometric to a product with unbounded factors, then either $H$ is quasi-isometric to the Cayley graph of $G$ or $H$ has bounded diameter in the Cayley graph of $G$.  In the first case, we reach a contradiction with the fact that $H$ is infinite index, and in the second case we reach a contradiction with the fact that $H$ is infinite.  Then $G$ has uniform exponential growth by \Cref{thm:mainthm}.
\end{proof}

\coroNonelem

\begin{proof}[Proof of \Cref{cor:CSnonelem}]
Let $(G,\s)$ be a virtually torsion-free hierarchically hyperbolic group such that $\fontact S$ is a non-elementary hyperbolic space.  The result follows immediately from  \cite[Theorem~9.14]{DurhamHagenSisto:Boundaries} and \Cref{prop:SinB}.
\end{proof}

We now turn our attention to the quantitative Tits alternative described in \Cref{thm:qtits}.  Under the additional assumption that $(G, \s)$ is hierarchically acylindrical our proof of \Cref{thm:technicalmainthm} can be adjusted to generate free subgroups rather than free semigroups. Hierarchical acylindricity was introduced by Durham, Hagen, and Sisto in \cite{DurhamHagenSisto:Boundaries}
to generalize the following property of mapping class groups: for any subsurface $\Sigma \subseteq S $,  the subgroup $MCG(\Sigma) \leq MCG(S)$ acts acylindrically on domains corresponding to $\Sigma$.  

	To make this precise in the HHG setting, let \[\stab{U} = \braces{g \in G : g^{\diamond}U = U}.\] 
	By definition of HHG, \(\stab{U}\) acts on \(\fontact U\). Let \(K_U\) be the kernel of the action, namely the subgroup \(\{g \in \stab{U} \mid g. x = x \quad \forall x \in \fontact U\}\).
	
\begin{defn}
	A hierarchically hyperbolic group is  \emph{hierarchically acylindrical} if $\stab{U}/ K_U$ acts acylindrically on \(\fontact U\), for all \(U \in \mf{S}\).  
\end{defn}
	
	For example, mapping class groups are hierarchically acylindrical because reducible subgroups of the mapping class group act acylindrically on the curve graph corresponding to a subsurface. Similarly, right-angled Artin groups are also hierarchically acylindrical because parabolic subgroups act acylindrically on the contact graph corresponding to the associated subgraph of the defining graph. 

However, not all hierarchically hyperbolic group structures are hierarchically acylindrical. An example is given by the group constructed by Burger--Mozes (see Example \ref{ex:BM}); for further discussion see \cite{DHS:Corrigendum}.

\propQuanTits

\begin{proof}[Proof of \Cref{thm:qtits}]
	Fix constants as in the proof of \Cref{thm:technicalmainthm}.  The only time that free semigroups are produced in the proof of \Cref{thm:technicalmainthm} is when Item \ref{item: combinatorial cases, fixed set} of \Cref{prop: basic combinatorics} holds and $\fontact S$ is an elementary hyperbolic space.  Equivalently, this occurs when two elements have independent axes in an infinite diameter domain that properly nests into $S$.  In this case, we pass to a subgroup $\hat G$ with finite generating set $Y$ which fixes $\overline{\mc B}$ pointwise, and find elements $s,t \in Y$ such that $s$ and $t^{-1}st$ are independent loxodromic isometries of $\fontact U$ for some $U\propnest S$. 
	By hierarchical acylindricity, $\hat G / K_U$ acts nonelementarily and acylindrically on $\fontact U$. Let \(\bar{s}\) and \(\bar{t}\) be the images of \(s\) and \(t\) in the quotient. 
	Applying \Cref{Fujiwara:freeSubgroup}, there exists a constant $k_5$ such that $\abrackets{\bar{s}^{k_5}, \bar{t}\bar{s}^{k_5}\bar{t}^{-1}} \cong \F_2$ in \(\hat G / K_U\).  Since free groups are Hopfian, this lifts to a free subgroup of $\hat{G}$.  In particular, the constant $M$ in  \Cref{thm:technicalmainthm} can be updated to be
	\[
		M = M \geq \max\{k_1, 2n_0+k_2, k_3+2,3(k_5+2)(N+1)!\}.
		\qedhere
	\]
\end{proof}

\begin{remark}
The proof of  \Cref{thm:qtits} shows that the conclusion of  \Cref{thm:qtits} also holds in slightly more generality.  In particular, it holds for any virtually torsion-free HHG in which Item \ref{item: combinatorial cases, find non orthogonal} of  \Cref{prop: basic combinatorics} holds for every finite generating set $X$.
\end{remark}

\renewcommand{\bar}{\overline}
\newcommand*{\C}{\mathcal{C}}
\newcommand*{\lhalf}[1]{\overleftarrow{#1}}

\newcommand{\hp}[1]{\marginpar{\color{ForestGreen}\tiny #1 -- hp}}

%
\appendix

\section{Uniform exponential growth for cocompactly cubulated groups \\  {\small By Radhika Gupta and Harry Petyt}}

The aim of this appendix is to show that cubical groups that admit a \emph{factor system} (see Section~\ref{section:prelim}) either have uniform exponential growth or are virtually abelian. This extends the work of the main body of the paper, in which this is shown to hold when the cube complex does not split as a direct product (Theorem \ref{thm:mainthm}). The arguments here show that this irreducibility assumption can be dropped. 

\begin{theorem} \label{thm:main}
Let $G$ be a group virtually acting freely cocompactly on a locally finite, finite dimensional CAT(0) cube complex $X$, and assume that $X$ has a factor system. Then either $G$ has uniform exponential growth or $G$ is virtually abelian. 
\end{theorem}

Since the arguments of Section \ref{sec:proof} require that the groups involved are (virtually) torsion-free, we consider free actions only, as any proper action of a torsion-free group on a CAT(0) cube complex is free.
\medskip

The class to which Theorem~\ref{thm:main} applies is very large---it is conjectured that all cocompact cube complexes have factor systems \cite{HagenSusse}. More specifically, it includes: compact special groups \cite{BehrstockHagenSisto:HHS1}; the Burger-Mozes group, and more generally BMW-groups in the sense of \cite{Caprace}; certain Artin groups \cite{haettel:virtually}; and any graph product of these \cite{BerlyneRussel:GraphProd}. We get a new proof of uniform exponential growth for some of these groups:

\begin{itemize}
\item   \emph{Compact special groups}. These groups are linear \cite{haglundwise:special,davisjanuszkiewicz:right}, and hence satisfy a UEG alternative by \cite{EMO} and the Tits alternative \cite{sageevwise:tits}.
\item   \emph{The mapping class group of the genus-two handlebody}. It is a subgroup of the mapping class group of the closed surface of genus two, which is linear by \cite{bigelowbudney, korkmaz:linearity}, and hence has UEG by \cite{EMO}. Note that in this case, linearity didn't come from virtual specialness. This group acts geometrically on a CAT(0) cube complex with a factor system by \cite{miller:stable}.
\item   \emph{Burger--Mozes group}: It acts freely on a 2-dimensional CAT(0) cube complex, so a UEG alternative follows by \cite{KarSageev}. It also appears as Example \ref{ex:BM}, but rather more machinery is needed there. Our proof does not depend on the 2--dimensionality, and we need little technology.
\end{itemize}

\subsubsection*{Acknowledgements}
We are very grateful to Mark Hagen for suggesting the strategy to prove Lemma~\ref{lem:nonelementary}. 

\subsection{Preliminaries} \label{section:prelim}
Let us briefly record a few definitions and recall some lemmas. 

\begin{lemma} \label{lem:product_ueg}
Let $G = G_1 \times G_2$ be a finitely generated group. If $G_1$ has uniform exponential growth, then $G$ has uniform exponential growth. 
\end{lemma}

\begin{proof}
This is a special case of the simple fact that if $G\to H$ is a surjective homomorphism and $H$ has uniform exponential growth, then so does $G$.
\end{proof}
\subsubsection*{CAT(0) cube complexes}
We refer the reader to \cite{sageev:cat(0)} for an introduction to CAT(0) cube complexes and groups acting on them. We highlight two points for later use; also see \cite[\S2]{BehrstockHagenSisto:HHS1}. Firstly, each hyperplane $h$ has two associated combinatorial hyperplanes. These are parallel copies of $h$ in the carrier of $h$ that are as far apart as possible. In particular, they are subcomplexes, unlike $h$ itself. However, like $h$, they are convex. Secondly, if $Y$ is a convex subcomplex of a CAT(0) cube complex $X$ and $x\in X^{(0)}$, then the \emph{gate} of $x$ in $Y$, denoted $\mf g_Y(x)$, is the unique closest vertex of $Y$ to $x$.

The \emph{contact graph} of a CAT(0) cube complex $X$, denoted $\C X$, is defined to be the graph whose $0$--skeleton is the set of hyperplanes of $X$, with an edge $(h,h')$ whenever the carriers of $h$ and $h'$ intersect. See \cite{hagen:geometry,hagen:weak} for more information on contact graphs, including a proof that they are quasitrees, and, in particular, hyperbolic.

\begin{lemma} \label{lem:separation}
Let $h_1$, $h_2, h_3$ be distinct hyperplanes of a CAT(0) cube complex $X$, such that $h_2$ separates $h_1$ from $h_3$. Then any path $P$ in $\C X$ from $h_1$ to $h_3$ passes through the 1--neighbourhood of $h_2$.
\end{lemma}

\begin{proof}
The hyperplanes $h_1$ and $h_3$ lie in different components of $X\smallsetminus h_2$, so $P$ must contain a hyperplane whose carrier intersects the carrier of $h_2$. 
\end{proof}

Note that an action on a CAT(0) cube complex induces an action on the contact graph.

\begin{lemma}[Double-skewering, {\cite[p.853]{capracesageev:rank}, \cite[Lem.~2.11, Cor.~4.5]{hagen:large}}] \label{lem:skew}
Suppose $G$ acts essentially cocompactly on a locally finite, finite dimensional CAT(0) cube complex $X$. If $w$ and $v$ are disjoint hyperplanes with halfspaces $\lhalf w\subsetneq\lhalf v$, then there is a hyperbolic isometry $g\in G$ such that $\lhalf{gv}\subsetneq\lhalf w$. Moreover, if $\dist_{\C X}(w,v)>2$, then $g$ can be taken to act loxodromically on $\C X$.
\end{lemma}

The contact graph comes with a coarsely Lipschitz projection map from $X$ to $\C X$ that sends each point $x\in X$ to the diameter--$1$ subset of $\C X$ consisting of all hyperplanes whose carriers contain $x$. More generally, if $Y$ is a convex subcomplex of $X$, then projection from $X$ to $\C Y$ is defined as the composition of the projection $Y\to\C Y$ with $\mf g_Y$.

\subsubsection*{Factor systems} 
A \emph{factor system} for a CAT(0) cube complex $X$ is the data of a collection of convex subcomplexes satisfying certain conditions. Factor systems were introduced in \cite{BehrstockHagenSisto:HHS1}, and any factor system for $X$ gives it the geometry of a hierarchically hyperbolic space. In general, $X$ may have many different factor systems, which allow for greater flexibility. However, if $X$ has a factor system, then it always has a ``simplest'' factor system, called the \emph{hyperclosure} \cite{HagenSusse}. Therefore, when $X$ has a factor system we shall assume that it is the hyperclosure of $X$. 

\begin{definition}[Hyperclosure]
For a CAT(0) cube complex $X$, the hyperclosure of $X$ is the intersection $\cal F$ of all sets $\cal F'$ of convex subcomplexes of $X$ satisfying the following conditions.
\begin{itemize}
\item   $\cal F'$ contains $X$ and every combinatorial hyperplane of $X$. 
\item   If $Y,Y'\in\cal F'$, then $\mf g_Y(Y')\in\cal F'$.
\item   If $Y\in\cal F'$ and $Y'$ is parallel to $Y$, in the sense that any hyperplane crossing one crosses both, then $Y'\in\cal F'$.
\end{itemize}
\end{definition}

If there exists an $N>0$ such that for all $x \in X$, there are at most $N$ subcomplexes $F \in \cal F$ with $x \in F$, then the hyperclosure $\cal F$ is a factor system. Moreover, if $\cal F$ is not a factor system, then $X$ does not have any factor system \cite[Rem.~1.15]{HagenSusse}.  It is not known whether the hyperclosure can fail to be a factor system when $X$ admits a proper cocompact group action.

Given a factor system for $X$, there is an associated hierarchically hyperbolic structure on $X$. This structure includes a hyperbolic space $\hat\C Y$ for each convex subcomplex $Y$ in the factor system. The space $\hat\C Y$ is obtained from the contact graph $\C Y$ by coning off certain subgraphs, so any projection to $\C Y$ induces a projection to $\hat\C Y$. (Note that the notation here disagrees slightly with that of the main body of the paper, where the hyperbolic spaces associated to a hierarchically hyperbolic space are always denoted $\C Y$. For cube complexes it is standard to reserve that notation for contact graphs.)

When the factor system is the hyperclosure of $X$, the hyperbolic space $\hat\C X$ associated to $X$ is quasi-isometric to the contact graph $\C X$ \cite[Rem.~8.18]{BehrstockHagenSisto:HHS1}.  More generally, if $X=\prod_{i=1}^nX_i$ is a direct product, then elements $Y$ of the hyperclosure come in two types: either $Y$ is a parallel copy of a subcomplex of some $X_i$, or $Y$ is the direct product of (the elements of) a subset of $\{X_i:1\le i\le n\}$. In the latter case, $\hat\C Y$ is a cone-off of a bounded graph, and hence is bounded. In other words, the hierarchical structure of $X$ is obtained by taking the disjoint union of the structures of the $X_i$ and introducing a finite number of elements whose associated hyperbolic spaces are bounded.

\subsection{Main result}
In what remains we shall prove Theorem~\ref{thm:main}. We begin with a specialisation of Theorem \ref{thm:technicalmainthm} to a certain kind of cocompactly cubulated group. 

\begin{prop} \label{prop:ueg}
Let $X=\prod_{i=1}^n X_i$ be a direct product of finite dimensional CAT(0) cube complexes. If $X$ has a factor system and the contact graph of each $X_i$ is unbounded but not quasiline, then any group acting freely cocompactly on $X$ has uniform exponential growth. 
\end{prop}

\begin{proof}
As discussed in Section~\ref{section:prelim}, any element of the hyperclosure $\cal F$ of $X$ whose associated hyperbolic space is unbounded is a subcomplex of some $X_i$. Moreover, $\hat{\C} X_i$ is not a quasiline, as it is quasi-isometric to $\C X_i$. Any group $G$ acting properly cocompactly on $X$ by cubical automorphisms is a hierarchically hyperbolic group, and the hierarchy structure of $G$ is the same as that of $X$. Note that $\{X_1,\dots,X_n\}$ is necessarily $G$--invariant. 

Let $S$ be a generating set of $G$, and let $N=\dim X$. Recall from  Section \ref{sec:proof} that $\bar{\cal B}=S^N\cal B\subset\cal{ F}$, where $\mc{B}=\bigcup_{s\in S}\B(s)$. By  Propositions~\ref{prop: basic combinatorics}, \ref{case:TransverseDomains}, and~\ref{case:ProperNesting}, either $G$ has uniform exponential growth, or $\bar{\cal B}$ is a $G$--invariant set of subcomplexes such that $\prod_{B\in\bar{\cal B}}B$ is a subcomplex of $X$, and each $\hat\C B$ is unbounded. In this latter case, $\bar{\cal B}\subset\{X_1,\dots,X_n\}$ by Lemma~\ref{lem:infinitediameter}, so none of the $\hat\C B$ are quasilines. This suffices, because the argument of Theorem~\ref{thm:technicalmainthm} shows that if none of the $\hat\C B$ are quasilines, then $G$ has uniform exponential growth. 
\end{proof}

\begin{lemma} \label{lem:nonelementary}
Suppose that a group $G$ acts essentially and cocompactly on a locally finite, finite dimensional, irreducible CAT(0) cube complex $Y$. Then the contact graph $\C Y$ of $Y$ is unbounded, and if $\C Y$ is a quasiline, then so is $Y$.
\end{lemma}
\begin{proof}
The contact graph $\C Y$ is unbounded by \cite[Cor.~4.7]{hagen:large}. Suppose that $\C Y$ is a quasiline. Let $h$ and $h'$ be two hyperplanes of $Y$ that are at distance at least 3 in $\C Y$. The ``moreover'' statement of Lemma~\ref{lem:skew} provides a corresponding hyperbolic isometry $g\in G$ that acts loxodromically on $\C Y$. Perhaps after subdividing $Y$, an application of \cite[Thm~1.4]{haglund:isometries} shows that $g$ is combinatorially hyperbolic on $Y$, and is therefore rank-one by \cite[Thm~4.1]{hagen:large}. Let $A$ be an axis of $g$ in $Y$, and let $C$ be the cubical convex hull of $A$. By \cite[Lem.~4.8]{hagen:large}, $C$ is at finite Hausdorff distance from $A$, and thus it is a quasiline. Moreover, since $g$ is loxodromic on the quasiline $\C Y$, the projection of $C$ to $\C Y$ is coarsely onto. 

It suffices to show that $C=Y$. Since $C$ is convex, it is enough to prove that every hyperplane of $Y$ crosses $C$. Suppose $h_1$ is a hyperplane of $Y$ that does not cross $C$. By the ``implication $(3)\implies(4)$'' part of the proof of \cite[Thm~4.1]{hagen:large}, the gate $\mf g_C(h_1)$ of $h_1$ to $C$ has uniformly bounded diameter. Since the projection map $Y\to\C Y$ is coarsely Lipschitz and since $\C Y$ is a quasiline, there is a ball $B\subset \C Y$ of uniformly bounded diameter that contains the projection of $\mf g_C(h_1)$ and disconnects $\C Y$. 

Let $h_2$ and $h_3$ be hyperplanes of $Y$ that cross $C$, have $\dist_{\C Y}(h_i,B)>1+\diam(B)$, and lie on opposite sides of $B$ inside $\C Y$. Since $h_1$ does not cross $C$, it cannot separate $h_2$ and $h_3$. It follows from Lemma~\ref{lem:separation} that neither $h_2$ nor $h_3$ separates the other from $h_1$. Thus, the $h_i$ form a facing triple that are pairwise at distance greater than 3 in $\C Y$. By Lemma~\ref{lem:skew} (Double skewering), there exists an isometry $g$ that acts loxodromically on $\C Y$ and has the property that $h_1$ separates $g^nh_2$ from $h_2$, and hence from $h_3$, for all positive $n$. From Lemma~\ref{lem:separation}, we see that $\C Y$ must have at least three ends, a contradiction. Thus every hyperplane of $Y$ crosses $C$, so $Y=C$.
\end{proof}

We are now in a position to prove Theorem~\ref{thm:main}. We restate it for convenience, minus the word ``virtually'', as that generalisation is immediate from \cite{ShalenWagreich}.

\numberedtheorem{Theorem~\ref{thm:main}}{}{Let $G$ be a group acting freely cocompactly on a locally finite, finite dimensional CAT(0) cube complex $X$, and assume that $X$ has a factor system. Then either $G$ has uniform exponential growth or $G$ is virtually abelian}

\begin{proof}
By passing to the essential core, we may assume that $X$ is essential \cite[\S3.2]{capracesageev:rank}, which implies that the action of $G$ on $X$ is essential, by cocompactness. Now, by \cite[Prop.~2.6]{capracesageev:rank}, there is a decomposition $X=X^1\times X^2$, where $X^1=\prod_{i=1}^nX_i$ is a product of irreducible non-euclidean complexes, and $X^2$ is euclidean. Since $X$ is essential, if $X^1$ is bounded then it is trivial. In this case, $G$ acts properly cocompactly on a flat, and so is virtually abelian.

Otherwise, \cite[Cor.2.8]{nevosageev:poisson} provides a finite index subgroup of $G$ that splits as $G_1\times G_2$, where $G_j$ acts properly cocompactly on $X^j$. By Lemma~\ref{lem:product_ueg}, it suffices to show that $G_1$ has uniform exponential growth, and by \cite{ShalenWagreich} we may pass to a finite index subgroup $G_1'$ of $G_1$ that fixes the factors of $X^1$. According to Lemma~\ref{lem:nonelementary}, the contact graph of each $X_i$ is unbounded but not a quasiline. The conditions of Proposition~\ref{prop:ueg} are therefore met by $X^1$, on which $G_1'$ is acting properly cocompactly. Thus $G_1'$ has uniform exponential growth, completing the proof.
\end{proof}



\bibliographystyle{alpha}
\bibliography{researchbib}


\end{document}